\documentclass[12pt]{amsart}

\usepackage{amsfonts}
\usepackage{amscd}
\usepackage{amssymb}

\allowdisplaybreaks

\usepackage{enumerate,enumitem,hyperref,cleveref}

\date{\today}

\newtheorem{thm}{Theorem}[section]
\newtheorem{cor}[thm]{Corollary}
\newtheorem{lem}[thm]{Lemma}

\newtheorem{prop}[thm]{Proposition}
\theoremstyle{definition}

\theoremstyle{remark}
\newtheorem{rem}[thm]{Remark}

\numberwithin{equation}{section}

\newcommand{\R}{\mathbb R}

\newcommand{\He}{\mathbb H}

\newcommand{\C}{{\mathbb C}}

\newcommand{\Fs}{\mathcal{F}^\lambda(\C^{2n})}

\renewcommand{\Re}{\operatorname{Re}}
\renewcommand{\Im}{\operatorname{Im}}
\newcommand{\tr}{\operatorname{tr}}

\usepackage{color}

\title[Modulation spaces on the  Heisenberg group]
{Modulation spaces on the Heisenberg group}

\author[  Biswas and Thangavelu]{  Md Hasan Ali Biswas and Sundaram Thangavelu}

\address[] {Department of Mathematics, Indian Institute of Science, Bangalore--560012, India.}
\email{mdhasanalibiswas4@gmail.com,\, veluma@iisc.ac.in}

\begin{document}

\maketitle

\vskip0.25in

\begin{abstract} In this article we show how certain irreducible unitary representation $ \Pi_\lambda $ of the twisted  Heisenberg group $ \He_\lambda^n(\C)$ leads to the  twisted modulation spaces $ M_\lambda^{p,q}(\R^{2n}).$  These $ \Pi_\lambda $ also turn out to be  irreducible unitary representations of another nilpotent Lie group $ G_n $ which contains two copies of the Heisenberg group $ \He^n.$ By lifting $ \Pi_\lambda $ we obtain another unitary  representation  $ \Pi $ of $ G_n $ acting on $ L^2(\He^n).$ We define our modulation spaces $ M^{p,q}(\He^n) $ in terms of the matrix coefficients associated to $ \Pi.$  These spaces are shown to be invariant under Heisenberg translations and Heisenberg modulations which are different from euclidean modulations. We also establish some of the basic properties of $ M_\lambda^{p,q}(\R^{2n})$ and $ M^{p,q}(\He^n) $ 
such as completeness and invariance under suitable Fourier transforms.
 \end{abstract}


\section{Introduction} \label{Sec-intro} 
We begin by recalling the definition of the modulation spaces $ M^{p,q}(\R^n) $ introduced by Feichtinger in \cite{fe83-4} and a detailed set of results were published in \cite{HGF}. The left regular representation of $ \R^n $ acting on $ L^2(\R^n) $ is the translation operator $ \tau(x) $ whose action is simply given by $ \tau(x)f(\xi) = f(\xi-x) $ for $ f \in  L^2(\R^n).$ By conjugating with the Fourier transform $ \mathcal{F} $ which acts unitarily on $ L^2(\R^n) $ we get another representation $  e(y) = \mathcal{F}^\ast \circ \tau(y) \circ \mathcal{F} $  which are the modulation operators with action $ e(y)f(\xi) = e^{i y\cdot \xi} f(\xi).$ These two families of operators do not commute but lead to the Heisenberg group  $ \He^n = \R^n \times \R^n \times \R $ whose group law is
\begin{equation} (x,y,t)(u,v,s) = \big(x+u, y+v, t+s+\frac{1}{2}(u \cdot y-v \cdot x)\big).
\end{equation}
The Schr\"odinger representation $ \pi(x,y,t) $ of $ \He^n$ acting on $ L^2(\R^n) $ is defined by 
\begin{equation}
 \pi(x,y,t)\varphi(\xi) = e^{it}\, e^{i(x\cdot \xi +\frac{1}{2} x\cdot y)}\,\varphi(\xi+y),\, \, \varphi \in L^2(\R^n). 
 \end{equation}
We observe that  this family $ \pi(x,y,t) $ contains both translations and modulations. Recall that the modulation space $  M^{p,q}(\R^n), 1 \leq p,q \leq \infty $  is defined to be the class of tempered distributions $ f $ for which  the Fourier-Wigner transform  $ V_gf(x,y) = \langle f, e(y)\circ \tau(x) g \rangle $ belongs to $ L^{p,q}(\R^n \times \R^n) $
for some Schwartz function $ g \in \mathcal{S}(\R^n).$ 
 In terms of the operator $ \pi(x,y)=\pi(x,y,0) $ we  can also express $ V_gf $ as 
$ V_gf(x,y) =  e^{-\frac{i}{2} x\cdot y}\, \langle f, \pi(y,-x)g \rangle.$
Consequently, $ f \in M^{p,q}(\R^n) $ if and only if the matrix coefficient $ \langle f, \pi(y,-x)g \rangle \in  L^{p,q}(\R^n \times \R^n)  $ for some $ g \in \mathcal{S}(\R^n).$\\

It is well known that the above definition of $ M^{p,q}(\R^n) $ is independent of the `window function' $ g.$  When  $ g $ is chosen to be the standard Gaussian $ e^{-\frac{1}{2} |x|^2} $ it turns out that  
$$ V_gf(x,y) e^{\frac{i}{2}x\cdot y} = Bf(x-iy) e^{-\frac{1}{4}(|x|^2+|y|^2)} $$
where $ B $ is the Bargmann transform which takes $ L^2(\R^n) $ unitarily onto the Fock space $ \mathcal{F}(\C^n).$
Recall that $ \mathcal{F}(\C^n)$ is the weighted Bergman space of entire functions corresponding to the weight $ w(x+iy) = e^{-\frac{1}{2}(|x|^2+|y|^2)}.$ Thus $ f \in M^{p,q}(\R^n) $ if and only if the function 
$$ F(x,y) = Bf(x+iy) \sqrt{w(x+iy)} \in  L^{p,q}(\R^n \times \R^n) .$$  We have a perfect analogy when $ \He^n $ is  replaced by the twisted Heisenberg group $ \He_\lambda^n(\C)$ leading to the twisted modulation spaces 
$ M_\lambda^{p,q}(\R^{2n}).$ \\

The left regular representation  $ \He^n $ acting on $ L^2(\He^n)$ gives rise to a family of twisted translations $ \tau_\lambda(\xi), \xi \in \R^{2n} $ for each non-zero $ \lambda \in \R.$ These are unitary operators acting on $ L^2(\R^{2n}) $ whose action can be transferred to the twisted Fock spaces $ \Fs$ studied in \cite{GT, KTX}. This is achieved by means of the twisted Bargmann transform $ B_\lambda $ since $ \Fs $ is exactly the image of $ L^2(\R^{2n}) $ under $ B_\lambda.$  The resulting unitary operator $ \rho_\lambda(\xi) $ initially defind for $ \xi \in \R^{2n} $ has a natural extension to $ \C^{2n}.$ Thus for each $ \zeta \in \C^{2n}$ we have a unitary operator $ \rho_\lambda(\zeta) $  on $ \Fs.$ For $ \eta \in \R^{2n},$ the operators $ \rho_\lambda(\eta)$ and $ \rho_\lambda(i\eta) $ are unitarily equivalent:  there is a unitary operator $ U $ on $ \Fs $ so that $ \rho_\lambda(i\eta) = U^\ast \circ \rho_\lambda(\eta) \circ U.$\\

   By defining $ \Pi_\lambda(\zeta) = B_\lambda^\ast \circ \rho_\lambda(\zeta) \circ B_\lambda $ we get a family of unitary operators acting on $ L^2(\R^{2n}).$  For $ \xi \in \R^{2n}$ it happens that  $ \Pi_\lambda(\xi) = \tau_\lambda(\xi) $ and by defining $ e_\lambda(\eta) = \Pi_\lambda(i\eta),\, \eta \in \R^{2n} $ we get another family of unitary operators on $ L^2(\R^{2n}).$ 
In analogy with translations and modulations on $ \R^n $ leading to $ \He^n,$  the operators $ \tau_\lambda(\xi) $ and $ e_\lambda(\eta)$ lead to  the twisted Heisenberg group 
$ \He_\lambda^n(\C) = \C^{2n} \times \R $ introduced by one of us in \cite{ST-IJPAM}, see Section 3.1.  The group law of $ \He_\lambda^n(\C) $ is given by
$$ (\zeta,t)(\zeta^\prime,t^\prime) = \big(\zeta+\zeta^\prime,t+t^\prime+\frac{\lambda}{2} (\coth \lambda)\, \Im (\zeta \cdot \overline{\zeta^\prime})-\frac{\lambda}{2}\, \Re [\zeta, \overline{\zeta^\prime}]\big).$$ 
see Section 2.3  for the notation. In terms of $ \Pi_\lambda(\zeta)$ we can build a unitary representation $ \Pi_\lambda(\zeta, t) $ of the twisted Heisenberg group.\\\

 We will show that the  unitary operator $ U_\lambda = B_\lambda^\ast \circ U \circ  B_\lambda $  intertwines $ e_\lambda(\eta) $ and $ \tau_\lambda(\eta)$: thus $ e_\lambda(\eta) = U_\lambda^\ast \circ \tau_\lambda(\eta) \circ U_\lambda.$ This is reminiscent of the well known relation $ e(y) = \mathcal{F}^\ast \circ \tau(y) \circ \mathcal{F} $ where $ \mathcal{F} $ is the Fourier transform on $ L^2(\R^n).$ In view of this, the operators $ e_\lambda(\eta) $ are qualified to be called twisted modulations. In this article we are interested in knowing what kind of spaces we will end up with if translations and modulations on $ \R^n $ are replaced by twisted translations $ \tau_\lambda(\xi) $ and twisted modulations $ e_\lambda(\eta) $ acting on $ L^2( \R^{2n}) $ for a given non-zero $ \lambda \in \R.$ We are led to the twisted modulation spaces  $M_\lambda^{p,q}(\R^{2n}) $ consisting  of  tempered distributions  $ f $ on $ \R^{2n} $ for which
the matrix coefficients $ \langle f, \Pi_\lambda(\xi+i\eta)g \rangle $ belong to  the mixed norm space $ L^{p,q}(\R^{2n} \times \R^{2n})$ for some Schwartz function $ g $ on $ \R^{2n}.$\\

As in the classical setting, the above definition can be shown to be independent of the window function $ g.$ Moreover, when we choose 
$ g(x,y) = p_{1/2}^\lambda(x,y) $ where $ p_t^\lambda $ is the heat kernel associated to the special Hermite operator $ L_\lambda $ on $ \R^{2n}$ it turns out that 
\begin{equation}
 d_n\, \sqrt{c_\lambda}\, \langle f, \Pi_\lambda(\xi+i\eta)g \rangle = B_\lambda f(\xi+i\eta)\, \sqrt{w_\lambda(\xi+i\eta)} 
 \end{equation}
where $ c_\lambda = (4\pi)^{-n} \lambda^n (\sinh \lambda)^{-n} $ and $ d_n $ is a constant. Here $ w_\lambda $ is the weight function appearing in the definition of $\Fs.$ In view of this we note that $ f \in M_\lambda^{p,q}(\R^{2n}) $ if and only if  
$$B_\lambda f(\xi+i\eta)\, \sqrt{w_\lambda(\xi+i\eta)}  \in  L^{p,q}(\R^{2n} \times \R^{2n}).$$ 
In this article we study some of the basic properties of these modulation spaces.
 Another goal is to give a reasonable definition of modulation spaces on the Heisenberg group. In order to define them we make use of the twisted modulation spaces.\\


 The family of unitary operators $ \Pi_\lambda(\zeta) $ acting on $ L^2(\R^{2n})$
satisfy the composition formula
\begin{equation}\label{Pi_lambda composition formula intro} \Pi_\lambda (\zeta)\, \Pi_\lambda(\zeta^\prime) =  \Pi_\lambda(\zeta+\zeta^\prime) e^{-i \frac{\lambda}{2} ((\coth \lambda)\, \Im (\zeta \cdot \overline{\zeta^\prime})- \Re [\zeta, \overline{\zeta^\prime}])}
\end{equation}
where $ [\zeta, \overline{\zeta^\prime}]$ is the obvious extension to $ \C^{2n}$ of the symplectic form on $ \R^{2n}.$
This suggests that we introduce a group structure on $ \C^{2n+1} = \C^{2n} \times \C $ by the rule
$$ (\zeta, s) (\zeta^\prime,s^\prime) =  (\zeta+\zeta^\prime, s+s^\prime +\frac{1}{2}\Re [\zeta, \overline{\zeta^\prime}]+\frac{i}{2} \Im (\zeta \cdot \overline{\zeta^\prime})) .$$
We denote this group by $ G_n $ and note that $ G_n $ contains two copies of $ \He^n $ as subgroups.  We get a family of unitary representations of $ G_n $ by defining
$$  \Pi_\lambda (\zeta, s) = e^{i \lambda( \Re s -(\coth \lambda)\Im s)} \,  \Pi_\lambda (\zeta).$$
These representations  are  irreducible as the family $ \rho_\lambda(\zeta) $ is irreducible as shown in \cite{GT}. \\

The operators $ \Pi_\lambda(\zeta,s) $ initially defined on $ L^2(\R^{2n})$ can  be lifted to a unitary operator  $\Pi(\zeta,s)$  acting on $ L^2(\He^n).$  When   $ \xi \in \R^{2n}$ it can be shown that 
$ \Pi(\xi,s) = \tau(\xi,s) $ the left translation on the Heisenberg group.
 There is a unitary operator   $ U: L^2(\He^n) \rightarrow L^2(\He^n)$  which intertwines $ \Pi_\lambda(\eta,s) $ with $ \Pi_\lambda(i\eta,s) $  for $ \eta \in \R^{2n}.$ Thus the operator $ e(\eta,s) = \Pi_\lambda(i\eta,s) $ satisfies  $ e(\eta,s) = U^\ast \circ \tau(\eta,s) \circ U.$ For this reason we can consider $ e(\eta,s) $ as analogues of modulations on the Heisenberg group. 
 We look for Banach spaces of distributions on $ \He^n$ that are invariant under the Heisenberg  translations and modulations or more generally under $ \Pi(\zeta,s) $ for $ s \in \R.$ \\
 
 The representation $ \Pi(\zeta,s) $ fails to be square integrable but  by considering a   modified matrix coefficient we get an interesting identity. Indeed, we can show that  for $ f \in L^2(\He^n),$
 $$ \label{matrix-planch} \int_{\C^{2n}} \int_{-\infty}^\infty |\langle Tf, \Pi(\zeta,s) p_{1/2} \rangle|^2\, d\zeta\, ds  =  c_n\, \int_{\He^n} |f(h)|^2 \,dh.$$
 Here $ p_t $ is the heat kernel associated to the sublaplacian on $ \He^n $ and $ T $ is a certain Fourier multiplier operator acting on the central variable.
 The above identity follows from the relation 
 $$ d_n\,  \int_{-\infty}^\infty e^{i\lambda s} \, \langle Tf, \Pi(\zeta,s) p_{1/2} \rangle \, ds=  B_\lambda f^\lambda(\zeta)\,\sqrt{w_\lambda(\zeta)} $$   Taking the clue from the equivalent definitions of $M^{p,q}(\R^n)$ and $ M^{p,q}_\lambda(\R^{2n})$ in terms of the Bargmann/twisted Bargmann transforms we make the following definition. We say that a tempered distribution $ f $ on $ \He^n$ belongs to $ M^{p,q}(\He^n) $ if and only if for every $ \lambda \neq 0,$  its inverse Fourier transform  $ f^\lambda $ in the central variable  belongs to $  M^{p,q}_\lambda(\R^{2n})$  in such a way that
$$ \| f\|_{(p,q)}^2 =  \int_{-\infty}^\infty \| \sqrt{w_\lambda(\cdot)}\, B_\lambda f^\lambda(\cdot) \|_{L^{p,q}(\R^{2n}\times \R^{2n})}^2 \, d\lambda < \infty.$$
The logic behind  this particular definition is explained in  Section 4.3, see also Remark \ref{rem}. We will show that these are Banach spaces invariant under the action of $ \Pi(\zeta,s), s \in \R $ and also under $ U $ when $ p =q.$ Some more basic properties of these spaces will also be established.\\

Many important function spaces, such as modulation spaces on $ \R^n $ or Bergman spaces on the unit ball in $ \C^n $ occur as coorbit spaces defined in terms of matrix coefficients associated to irreducible unitary representations of certain Lie groups. In particular, the case of nilpotent Lie groups have received considerable attention in the works of  Beltita and Beltita  \cite{BB1, BB2} and Grochenig \cite{KG1}.  In these works  we find detailed study of modulation spaces associated to representations  of nilpotent Lie groups which are square integrable modulo the center.  Some of  our results on twisted modulation spaces  can be read out from these papers, e.g Theorem 3.3  from \cite{BB1} and Theorem 3.7 from \cite{BB2}.  However, our motivation for the study of twisted modulation spaces  comes from twisted Fock spaces and a search for a reasonable definition of modulation spaces on Heisenberg  groups.\\

We are not the first to consider modulation spaces on the Heisenberg group. We would like to bring the work of Fischer-Rottensteiner-Ruzhansky \cite{FRR} to the attention of the readers. In this very interesting work the authors have defined and studied modulation spaces on the Heisenberg group which they denote by $ E^{p,q}(\R^{2n+1}) $ in the unweighted case. These spaces are defined in terms of the matrix coefficients associated to a square integrable representation of the Dynin-Folland group $ \He_{n,2} $ in very much the same way $ M^{p,q}(\R^n)$ are defined in terms of the  Schr\"odinger representation  $ \pi$ of $ \He^n.$  This group studied by Dynin \cite{ASD} and Folland \cite{GBF1} is the semidirect product of  $ \He^n $ and $ \R^{2n+2}$ which contains $ \He^n $ as a subgroup. Compare this with the fact that $ \He^n $ can be considered as the semidirect product of  $ \R^n$ and $ \R^{n+1}.$  As a set $ \He_{n,2} $ can be identified with $ \He^n \times \He^n \times \R $ whose  elements are  written as $ (P,Q,S).$ \\

There is a family of irreducible unitary representation $ \pi_\lambda $ of $ \He_{n,2}$ whose action on $ L^2(\He^n) $ is explicitly given by 
$$ \pi_\lambda(P,Q,S)f(X) = e^{2\pi i \lambda( S+(Q,X))}\, f(XP).$$
The similarity between the Schr\"odinger representation is obvious. These  representations turn out be square integrable modulo the centre and hence there is a canonical way of defining the modulation spaces. Note that $ \pi(P,0,0) $ is the right translation by elements of  $ \He^n $ and $ \pi(0,Q,S) $ is the usual modulation on $ \R^{2n+1}.$  Thus the modulation spaces $ E^{p,q}(\R^{2n+1}) $  are invariant under the Heisenberg translations and euclidean modulations. Though at present we do not know how to compare these modulation spaces  with our $ M^{p,q}(\He^n),$ the following remarks are in order. \\

Despite there is an underlying nilpotent group $ G_n $ in our definition, we are motivated by the twisted modulation spaces which are associated to a square integrable representation $ \Pi_\lambda $ of  the twisted Heisenberg group $ \He^n_\lambda(\C).$  Even though the representation $ \Pi $ of $ G_n $ is built in terms of $ \Pi_\lambda$ it is not square integrable even if we go modulo the centre. In defining $ M^{p,q}(\He^n) $ we have used the modified matrix coefficients of $ \Pi $ restricted to $ \C^{2n} \times \R \subset G_n.$ The resulting spaces are invariant under Heisenberg translations $ \tau(h) $ and Heisenberg modulations $e(h)$ which are different from the standard euclidean modulations. Roughly speaking, our modulation spaces are direct integrals of twisted modulation spaces.\\

The plan of the paper is as follows. After recalling the basic representation theory of the Heisenberg group, twisted Bargmann transform and twisted Fock space in Section 2,  we move on to section 3 where we introduce the twisted Heisenberg group $ \He^n_\lambda(\C) $ and  study the matrix coefficients of a representation $ \Pi_\lambda $  realised on $ L^2(\R^{2n}).$ Later we define the twisted modulation spaces $ M_\lambda^{p,q}(\R^{2n})$ and study some basic properties of the same. In Section 4 we define the group $ G_n $ and introduce the  representation $ \Pi $ which is realised on $ L^2(\He^n) .$ Based on properties of the matrix coefficients of $ \Pi $ we define the modulation spaces $ M^{p,q}(\He^n) $ and investigate some of  its basic properties.\\

We conclude this introduction with a warning about the convention we have followed about constants that are not important. Apart from the well known convention that constants $ C $ in various inequalities are not necessarily same, we also use, by abuse of notation,   the same constant $ c_n $ or $ d_n $ in several equations  to avoid keeping track of their not so important exact values.\\

\section{Heisenberg groups and twisted Heisenberg groups}

\subsection{Heisenberg group $\He^n$ and the representations $ \pi_\lambda$} We begin by recalling some well known facts about the Heisenberg groups $ \He^n $ and their representations. For more details we refer to  the monographs \cite{GBF, ST-uncertainty}. The group $  \He^n$ is  just $ \R^{2n} \times \R  =  \C^n \times \R $ equipped with the group law
$$ (z,t)(w,s) =  (z+w, t+s+\frac{1}{2} \Im(z \cdot \bar{w})).$$
Note that  if we let $ z = x+iy, w = u+iv,$ then $ \Im(z \cdot \bar{w}) =(u \cdot y-x \cdot v)$ is nothing but the symplectic form $ [\xi,\eta] $ on $ \R^{2n}$ between $ \xi = (x,u)$ and $ \eta =(y,v).$ Therefore, the group law can also be written in the form
$$ (\xi,t)(\eta,s) =(\xi+\eta, t+s+\frac{1}{2}[\xi,\eta]).$$
The group $ \He^n $ is a step two nilpotent Lie group, it is therefore unimodular and the Haar measure is simply given by the Lebesgue measure on $ \R^{2n+1}.$  The  spaces $ L^p(\He^n) $ are defined with respect to the Lebesgue measure. The convolution between two functions $ f, g \in L^1(\He^n) $ is defined in the usual way
$$ f \ast g(\xi,t) = \int_{\He^n} f ( (\xi,t)(\eta,s)^{-1}) \, g(\eta,s)\, d\eta\, ds.$$
The convolution on the Heisenberg group gives rise to a family of convolutions  for functions on $ \R^{2n} $ known as twisted convolutions.\\

To define them let us set up some notation.
For a function $ f \in L^1(\He^n) $ let us denote by $ f^\lambda (x,u)$ the inverse Fourier transform of $ f $ in the central variable: thus
$$ f^\lambda(\xi) = \int_{-\infty}^\infty  f(\xi,t)\, e^{i\lambda t}\, dt.$$
Since $ (\eta,s)^{-1} = (-\eta,-s) $ a simple calculation shows that
$$ (f \ast g)^\lambda(\xi) = \int_{\R^{2n}} f^\lambda(\xi-\eta)\, g^\lambda(\eta)\,e^{i\frac{\lambda}{2} [\xi,\eta]}\, d\eta .$$
The right hand side defines what is known as the $ \lambda$-twisted convolution of $ f^\lambda $ with $ g^\lambda $  denoted by $ f^\lambda \ast_\lambda g^\lambda.$
By defining  the $\lambda$-twisted translation $ \tau_\lambda(\eta) = \tau_\lambda(y,v) $ by the equation 
\begin{align} \label{def:twisted-translation}
  \tau_\lambda(\eta)g^\lambda(\xi) = g^\lambda(\xi-\eta)\, e^{-i\frac{\lambda}{2} [\xi,\eta]} 
 \end{align}
we see that  the twisted convolution $ f^\lambda \ast_\lambda g^\lambda $  is given by the following integral
\begin{equation}\label{t-con} f^\lambda \ast_\lambda g^\lambda(\xi) = \int_{\R^{2n}} f^\lambda(\eta)\,  \tau_\lambda(\eta)g^\lambda(\xi)\, d\eta. \end{equation}
Taking partial Fourier transform in the central variable is a useful technique which reduces problems on  $ \He^n$ to problems on $ \R^{2n}.$\\

We now briefly recall the representation theory of $ \He^n$ which is needed in order to define the group Fourier transform. For each $ \lambda \in \R, \lambda \neq 0,$ there is an irreducible unitary representation $ \pi_\lambda $ of $ \He^n$ on $ L^2(\R^n) $ explicitly given by (see \cite{GBF})
$$ \pi_\lambda(x,y,t)f(\xi) = e^{i\lambda t} e^{i\lambda(x \cdot \xi +\frac{1}{2} x\cdot y)} f(\xi+y),\,\,\, f \in L^2(\R^n).$$
We use the notation $ \pi_\lambda(x,y) $ for the operator $ \pi_\lambda(x,y,0)$ and when $ \lambda = 1$ we simply write $ \pi(x,y) $ instead of $\pi_1(x,y) .$ We observe that  $ \pi(x,y) $ is factored as 
$ \pi(x,y) = e^{\frac{i}{2} x\cdot y} \, e(x)\circ \tau(-y) $  where  $ \tau(y) $ and $ e(x) $ are the standard euclidean translation and modulation. The Fourier transform a  function $ f \in L^1(\He^n) $ is the operator valued function 
$$ \hat{f}(\lambda) =: \int_{\He^n} f(\xi,t) \, \pi_\lambda(\xi,t)\,\, d\xi\, dt.$$
Recalling the definition of $ \pi_\lambda $ and $ f^\lambda $ we see that $ \hat{f}(\lambda) = \pi_\lambda(f^\lambda) $ where 
$$ \pi_\lambda(f^\lambda) =: \int_{\R^{2n}} \, f^\lambda(\xi)\, \pi_\lambda(\xi)\, d\xi $$
is known as the Weyl transform of $ f^\lambda.$ Properties of the Fourier transform $ f \rightarrow \hat{f} $ on $ \He^n$ are proved by studying the Weyl transform.\\

The Weyl transform initially defined on $ L^1(\R^{2n}) \cap  L^2(\R^{2n}) $  has an extension to the whole of $ L^2(\R^{2n})$ as a Hilbert-Schmidt operator.
It takes $ L^2(\R^{2n}) $ onto the the space $ \mathcal{S}_2 $ of Hilbert -Schmidt operators acting on $ L^2(\R^n).$ The inversion formula for the Weyl transform reads as
\begin{equation}\label{weyl-inversion}
f(x,u) = (2\pi)^{-n/2}\, |\lambda|^{n/2}\, \tr (\pi_\lambda(-x,-u) \pi_\lambda(f)).
\end{equation}
It can be shown that  $ \pi_\lambda(f \ast_\lambda g) = \pi_\lambda(f)\, \pi_\lambda(g) .$  For easy reference we record the following relation between  the two families 
of unitary operators $ \tau_\lambda(a,b) $ and $ \pi_\lambda(x,u):$ 
$$ \pi_\lambda(a,b)\pi_\lambda(x,u) = \pi_\lambda(x+a,u+b) e^{-i\frac{\lambda}{2}(u\cdot a- x\cdot b)} $$
which follows from the fact that $ \pi_\lambda $ is a representation of the Heisenberg group $ \He^n.$ Another easy calculation shows that
$$ \pi_\lambda(a,b) \pi_\lambda(f) = \int_{\R^{2n}}  f(x,u) \pi_\lambda(x+a,u+b) e^{-i\frac{\lambda}{2}(u\cdot a- x\cdot b)} dx\,du .$$
Hence the Weyl transform of  the twisted translation of a function is given by the relation 
$ \pi_\lambda(a,b) \pi_\lambda(f) = \pi_\lambda( \tau_\lambda(a,b)f) $
which  translates into the property
\begin{equation}\label{rel-one}
 \tau_\lambda(a,b) ( f\ast_\lambda g) = \tau_\lambda(a,b)f \ast_\lambda g 
 \end{equation}
for  the $\lambda$-twisted convolution of $ f $ with $ g $ in view of  the relation $ \pi_\lambda(f \ast_\lambda g) = \pi_\lambda(f)\, \pi_\lambda(g) .$ \\

\subsection{The twisted Bargmann transform and the twisted Fock spaces}
As we have remarked in the introduction the classical modulation spaces can be defined in terms of the Bargmann transform $ B $ which takes $ L^2(\R^n) $ unitarily onto the Fock space $ \mathcal{F}(\C^n) $ consisting entire functions on $ \C^n$ that are square integrable with respect to the Gaussian measure. It is well known that Bargmann transform is closely related to the Schr\"odinger representation $ \pi $ of  $ \He^n.$ In a similar way, we define the twisted  modulation spaces in terms of a unitary representation of the twisted Heisenberg group. The motivation for this group comes from certain unitary representation of the twisted Fock space $ \Fs $ which is the image of $ L^2(\R^{2n}) $ under the twisted Bargmann transform $ B_\lambda.$ Hence we begin by recalling the necessary background material to state relevant results on $ \Fs.$ The basic reference for this section is \cite{KTX,GT, ST-IJPAM}. We also refer the reader to \cite{ST-princeton, ST-uncertainty} for results related to Hermite and special Hermite operators.\\

Consider the scaled Hermite operator $ H(\lambda) = -\Delta+\lambda^2 |x|^2 $ on $ \R^n$  which is a positive operator with an explicit spectral decomposition given in terms of the Hermite functions.  We 
let $ e^{-tH(\lambda)}, t >0$ stand for the  semigroup generated by $ H(\lambda). $  We also consider the special Hermite operator $ L_\lambda $ on $ \R^{2n} $ generating another semigroup $ e^{-t L_\lambda} $ with the associated kernel
$$ p_t^\lambda(x,u) = (4\pi)^{-n} \, \lambda^n (\sinh t\lambda)^{-n} e^{-\frac{1}{4}\lambda (\coth t\lambda)(|x|^2+|u|^2)} $$
which has  a holomorphic extension $ p_t^\lambda(z,w) $ to $ \C^{2n}.$  Let us write  $ c_\lambda = (4\pi)^{-n} \, \lambda^n (\sinh \lambda)^{-n} .$ The twisted Bargmann transform $ B_\lambda $ is defined on $ L^2(\R^{2n}) $ by
\begin{align} \label{def:Gauss-Bargmann-transform-tiwsted-fock}
B_\lambda f(z,w) = (2\pi)^{-n/2}\, |\lambda|^{n/2}\,  c_\lambda \,\, p_1^\lambda(z,w)^{-1}\, \tr\left( \pi_\lambda(-z,-w) \pi_\lambda(f) e^{-\frac{1}{2}H(\lambda)} \right) .
\end{align}
In view of the inversion formula \eqref{weyl-inversion} for the Weyl transform, using the fact that $ \pi_\lambda(p_t^\lambda) = e^{-tH(\lambda)} ,$  the twisted  Bargmann transform is also given by
$$ B_\lambda f(z,w) =  \,c_\lambda \,\, p_1^\lambda(z,w)^{-1}\,  f \ast_\lambda p_{1/2}^\lambda(z,w). $$
The image of $ L^2(\R^{2n}) $ under $ B_\lambda $ is known to be a weighted Bergman space with the explicit weight function
$$ w_\lambda(\zeta) = c_{\lambda}\, \, e^{-\frac{1}{2} \lambda (\coth \lambda)|\zeta|^2}\, e^{\lambda [ \Re \zeta,\,\Im \zeta ]},\,\,  \zeta = (z,w). $$
Here $ [\xi,\eta] $ is the symplectic form on $ \R^{2n} $ defined by $ [(x,u),(y,v)] = (u\cdot y-v\cdot x).$ \\

The twisted Fock space $ \Fs $ is defined as the space of all entire functions on $ \C^{2n} $ square integrable with respect to the measure $ w_\lambda(\zeta)\, d\zeta $ where $ d\zeta $ is the Lebesgue measure on $ \C^{2n}.$  When equipped with the inner product 
$$ \langle F, G \rangle = \int_{\C^{2n}} F(\zeta) \, \overline{G}(\zeta)\, w_\lambda(\zeta)\, d\zeta ,$$ 
the space $ \Fs $ becomes a reproducing kernel Hilbert space and the Bargmann transform
$ B_\lambda: L^2(\R^{2n}) \rightarrow \Fs $ becomes  a unitary operator. The reproducing kernel for $ \Fs$ is known explicitly and given by
\begin{equation}\label{rep-ker}
K_\zeta(\zeta^\prime) = d_n \, c_\lambda\, \,e^{ \frac{\lambda}{2} (\coth \lambda)\, (\overline{\zeta} \cdot \zeta^\prime)}\, e^{i\frac{\lambda}{2}\,  [ \overline{\zeta}, \zeta^\prime]}
\end{equation}
where $ d_n $ is a constant depending only on the dimension. We refer the reader to \cite{KTX, GT} for more details of this transform and further properties of $ \Fs.$ We note that when $ \lambda = 0 $ the twisted Fock space reduces to the classical Fock space $ \mathcal{F}(\C^{2n}).$\\

The operator $ U $ defined on $ \Fs $  by $ UF(\zeta) = F(-i\zeta)$  is a unitary operator since the weight function $ w_\lambda $ and the Lebesgue measure $ d\zeta $ are invariant under the map $ \zeta \rightarrow -i\zeta.$ Consequently, as $ B_\lambda $ is unitary, there exists a unitary operator $ U_\lambda $ on $ L^2(\R^{2n}) $ defined by the relation
$ U \circ B_\lambda = B_\lambda \circ U_\lambda.$ In \cite{GT} the authors have calculated the operator $ U_\lambda $ explicitly and shown that for  any $ f \in  L^2(\R^{2n}) $  
\begin{align} \label{intertwine}
U_{\lambda}f(x,u) = c(\lambda)^n \, \widehat{f}(c(\lambda)(x,u)) 
\end{align}
where $ c(\lambda)= \frac{1}{2} \lambda (\coth  \frac{1}{2}\lambda)$  and $\widehat{f}$ is the euclidean Fourier transform of $f$ on $\mathbb R^{2n}.$ The operator $U_\lambda $ becomes the Fourier transform on $ \R^{2n}$ when $ \lambda $ becomes $0.$ The function $ c(\lambda) $ occurs in many places  whose properties are investigated in Lemma \ref{clambda}.\\

\subsection{Twisted Heisenberg group and a unitary representation}  We begin by setting up some notations. The symplectic form on $ \R^{2n} $ defined by  $ [ (x,u), (y,v) ] = (u \cdot y-v \cdot x)$  has a natural extension to $ \C^{2n} $ given by 
$$[\zeta, \zeta^\prime] =[(z,w),(z^\prime,w^\prime)] = (w\cdot z^\prime- z \cdot w^\prime),\,\, \zeta =(z,w), \zeta^\prime =(z^\prime,w^\prime).$$
For any $ \lambda \neq 0,$ the twisted Heisenberg group $ \He^n_\lambda(\C) $ is defined as the manifold $ \C^{2n}\times \R $ equipped with the group law
$$ (\zeta,t)(\zeta^\prime,t^\prime) = \big(\zeta+\zeta^\prime,t+t^\prime+\frac{\lambda}{2} (\coth \lambda)\, \Im (\zeta \cdot \overline{\zeta^\prime})-\frac{\lambda}{2}\, \Re [\zeta, \overline{\zeta^\prime}]\big). $$
Note that when $ \lambda = 0$ the twisted Heisenberg group reduces to the standard Heisenberg group $ \He^{2n} = \C^{2n}\times \R$ with the group law
$$ (\zeta,t)(\zeta^\prime,t^\prime) = \big(\zeta+\zeta^\prime,t+t^\prime+\frac{1}{2} \, \Im (\zeta \cdot \overline{\zeta^\prime})\big). $$
 The twisted Heisenberg group is an example of a  Mackey group, see \cite{GM}. \\
 
 To every real $ d \times d $ skew-symmetric matrix $ \Theta $ there is an associated nilpotent Lie group $ \He_\Theta $ which is just  $ \R^d \times \mathbb{T} $ equipped with the following group law:
$$ (x,e^{it})(y,e^{is}) =( x+y, e^{i(s+t+\frac{1}{2} \langle x, \Theta y\rangle)}).$$ The universal covering of this group is $ G_\Theta = \R^d \times \R $ with the group law
$$ (x,s)(y,t) = (x+y, s+t+\frac{1}{2}\langle x, \Theta y\rangle).$$
Note that Heisenberg groups are of the form $ G_\Theta$ for a suitable $ \Theta.$ More generally, as proved by Arhancet et al in \cite{AHKP} $ G_\Theta $ is a Metivier group if and only if $ \Theta$ is invertible and a H-type group if and only if $\Theta $ is orthogonal.\\

Identifying $ \C^{2n} $ with $ \R^{2n} \times \R^{2n} $ and writing $ \zeta = \xi+i\eta,\, \xi=(x,u), \eta = (y,v)  $ the  group product $(\zeta,t)(\zeta^\prime,t^\prime)$ of two elements from $ \He^n_\lambda(\C) $ is given by
$$ \big(\xi+\xi^\prime,\eta+\eta^\prime,\, t+t^\prime+\frac{\lambda}{2} (\coth \lambda)( x^\prime \cdot y- y^\prime \cdot x +u^\prime \cdot v-v^\prime \cdot u)-\frac{\lambda}{2}\,(x^\prime \cdot u-u^\prime \cdot x +y^\prime \cdot v-v^\prime \cdot y) \big) .$$
In terms of    $ I_n ,$  the $ n \times n $ identity matrix, and the matrix 
$
\mathbf{J}=
\begin{pmatrix}
    0 & I_n \\
    -I_n & 0 \\
  \end{pmatrix}
$ we define  the matrix $ \Theta_\lambda $  by
$$ \Theta_\lambda=\lambda\begin{pmatrix}
\mathbf{J}& -(\coth\lambda)I_{2n}\\
(\coth\lambda)I_{2n}& \mathbf{J}
\end{pmatrix}. $$
Then $ \He^n_\lambda(\C) $  becomes a Mackey group since the group law can be written in   the form
$$ \big((\xi,\eta)+(\xi^\prime,\eta^\prime), t+t^\prime+ \frac{1}{2} \langle (\xi,\eta), \Theta_\lambda(\xi^\prime,\eta^\prime)\rangle \big).$$
It is clear that $ \Theta_\lambda $ is skew-symmetric and easy to check that it is invertible but not orthogonal. Thus $ \He^n_\lambda(\C) $ is a Metivier group but not of Heisenberg type.
\\

The twisted Fock space $ \Fs$ supports a representation $ \rho_\lambda(\zeta,t) $ of the twisted Heisenberg group which we describe now.  We consider the following family of unitary operators defined on the twisted Fock space by
\begin{equation}\label{fock-rep}
 \rho_\lambda(\zeta^\prime)F(\zeta) = F(\zeta-\zeta^\prime)\, e^{ \frac{\lambda}{2} (\coth \lambda)\, (\zeta \cdot \overline{\zeta^\prime})}\, e^{-i\frac{\lambda}{2}\,  [\zeta, \overline{\zeta^\prime}]}\, e^{-\frac{\lambda}{4} (\coth \lambda)|\zeta^\prime|^2}\,  e^{\frac{\lambda}{2} [ \Re \zeta^\prime,\,\Im \zeta^\prime ]}.
 \end{equation}
 As shown in \cite{GT} this family of operators satisfy the composition law 
 $$ \rho_\lambda(\zeta)\, \rho_\lambda(\zeta^\prime) = \rho_\lambda(\zeta+\zeta^\prime) \, e^{-i \frac{\lambda}{2} ((\coth \lambda)\, \Im (\zeta \cdot \overline{\zeta^\prime})- \Re [\zeta, \overline{\zeta^\prime}])}.$$
 By defining   $ \rho_\lambda(\zeta^\prime,t^\prime) = e^{-it^\prime}\, \rho_\lambda(\zeta^\prime) $ it is not difficult to show that $ \rho_\lambda $ is indeed a unitary representation of $ \He^n_\lambda(\C).$ In fact, the group law of $ \He_\lambda^n $  itself was suggested by the above composition law.
 In \cite{GT} the authors have shown that $ \rho_\lambda $ is an irreducible unitary representation of the twisted Heisenberg group, see also \cite{ST-IJPAM}.  As the twisted Fock space $ \Fs$ is isometrically isomorphic to $ L^2(\R^{2n}) $ this representation leads to another representation realised on $ L^2(\R^{2n}).$  The action of this representation which we denote by $ \Pi_\lambda $ can be calculated explicitly in terms of twisted translations and twisted modulations.

\section{Twisted modulations spaces} 

\subsection{ Twisted translations and modulations} We now  proceed to define twisted modulations. By conjugating $ \rho_\lambda $ with the unitary operator $ B_\lambda $ we get a  representation $ \Pi_\lambda $ of $ \He^n_\lambda(\C) $ realised on $ L^2(\R^{2n}).$ Thus we define $$ \Pi_\lambda(\zeta,t) = B_\lambda^\ast \circ \rho_\lambda( \zeta, t)
\circ B_\lambda.$$ Note that $\Pi_\lambda(\zeta,t)= e^{-it}\, \Pi_\lambda(\zeta) $ where $\Pi_\lambda(\zeta)=\Pi_\lambda(\zeta,0)$ which can be calculated explicitly. With $ \zeta = \xi+i\eta ,$ sometimes we use the notation $ \Pi_\lambda(\xi,\eta) $ for $ \Pi_\lambda(\xi+i\eta).$ An easy calculation shows that the operator $ \rho_\lambda(\zeta)= \rho_\lambda(\xi+i\eta) $
can be factored as
$$ \rho_\lambda(\xi+i\eta) = e^{i\frac{\lambda}{2} (\coth \lambda)\, \xi \cdot \eta}\, \rho_\lambda(i\eta)\, \rho_\lambda(\xi)\,  .$$ 
In view of this, conjugation with the twisted Bargmann transform $ B_\lambda $ leads to the relation 
\begin{equation}\label{pi-rep} \Pi_\lambda(\xi, \eta) = e^{i\frac{\lambda}{2} (\coth \lambda)\, \xi \cdot \eta}\, e_\lambda(\eta) \, \tau_\lambda(\xi)\, \end{equation}
where $ \tau_\lambda(\xi) = B_\lambda^\ast \circ \rho_\lambda(\xi) \circ B_\lambda $ and $ e_\lambda(\eta) = B_\lambda^\ast \circ \rho_\lambda(i\eta) \circ B_\lambda .$ These operators  acting on $ L^2(\R^{2n})$ can be calculated explicitly.\\

From  the definition of $ \rho_\lambda(\zeta) $ and $ U $ it is easy to check that  $ U \circ \rho_\lambda(\zeta)\circ U^\ast = \rho_\lambda(i\zeta)$ and hence recalling that $ U \circ B_\lambda = B_\lambda \circ  U_\lambda $ we obtain the relation
$$ \tau_\lambda(\eta) =  B_\lambda^\ast  \circ U^\ast \circ \rho_\lambda(i\eta) \circ U \circ B_\lambda  = U_\lambda^\ast \circ B_\lambda^\ast  \circ \rho_\lambda(i\eta) \circ  B_\lambda \circ U_\lambda.$$
This simply means that $ e_\lambda(\eta) $ and $ \tau_\lambda(\eta)$ are unitarily equivalent to each other:
\begin{equation}\label{trans-mod} e_\lambda(\eta) =  U_\lambda \circ \tau_\lambda(\eta) \circ  U_\lambda^\ast.
\end{equation}
 As we show below $ \tau_\lambda(\xi) $  turned out to be nothing but  the $\lambda$-twisted translation. Since $ e_\lambda(\eta) $ is unitarily equivalent to $ \tau_\lambda(\eta) $ we call them $\lambda$-twisted modulations.\\

\begin{prop} For $ f \in L^2(\R^{2n}) $ and $ (a, b) \in \R^{2n} $ we have the following:
\begin{align} \label{twisted-translation} 
B_\lambda^\ast \circ \rho_\lambda(a,b) \circ B_\lambda =\tau_\lambda(a,b),\,\,
\end{align}
is just the $\lambda$-twisted translation whereas the $\lambda$-twisted modulations are given by
\begin{align} \label{twisted-modlation} 
e_\lambda(a,b)f(x,u) = f\big(x-(\tanh \frac{ \lambda}{2})b,u+(\tanh \frac{\lambda}{2})a\big)\,e^{-i\frac{\lambda}{2}(\coth \frac{\lambda}{2})(x\cdot a+ u\cdot b)} .
\end{align}
\end{prop}
\begin{proof}  If we let $ F = B_\lambda f,$ then by  the inversion formula \eqref{weyl-inversion} we see that
$$ F(x,u) = c_\lambda\,  p_1^\lambda(x,u)^{-1}\, f \ast_\lambda p_{1/2}^\lambda(x,u).$$
Recalling the action of $ \rho_\lambda(a,b) $ on $ F $ we get
$$\rho_\lambda(a,b)F(x,u)=e^{-\frac{\lambda}{4}(\coth \lambda)(a^2+b^2)}\,e^{\frac{\lambda}{2}(a\cdot x+b\cdot u)} \, \tau_\lambda(a,b)F(x,u).$$
Using the relation \eqref{rel-one} we calculate that
$$ (\tau_\lambda(a,b)F(x,u) = p_1^\lambda(x-a,u-b)^{-1}\, p_1^\lambda(x,u)\, B_\lambda (\tau_\lambda(a,b)f)(x,u).$$
Using this in the expression for $(\tau_\lambda(a,b)F(x,u) $ and simplifying we get
$$\rho_\lambda(a,b)(B_\lambda f)(x,u)= B_\lambda(\tau_\lambda(a,b)f)(x,u).$$
We  calculate $ e_\lambda(a,b) =B_\lambda^\ast \circ \rho_\lambda(i(a,b)) \circ B_\lambda$ by  making use of the relation \eqref{trans-mod} verified above. As $ U_\lambda $ is essentially the Fourier transform, the  expression for $ e_\lambda(a,b) $ is obtained after an easy calculation.
\end{proof}

As the representation $ \Pi_\lambda$ on $ L^2(\R^{2n})$ is unitarily equivalent  to the irreducible representation $ \rho_\lambda $ on $ \Fs ,$ it is also irreducible. We now show that  it is square integrable modulo the center of the twisted Heisenberg group $ \He_\lambda^n(\C).$  The following proposition is also true in the more general context of Metivier groups, see \cite{BB1}.

\begin{prop}\label{square-int} For any $ f, g \in L^2(\R^{2n}),$ the matrix coefficient $\langle f, \Pi_\lambda(\xi,\eta)g \rangle $ belongs to $ L^2(\R^{2n} \times \R^{2n}).$ Moreover,
$$ \int_{\R^{2n}}\int_{\R^{2n}} |\langle f, \Pi_\lambda(\xi,\eta)g \rangle |^2\, d\xi\, d\eta =  d_n \, c_\lambda^{-2}   \left(\int_{\R^{2n}} |f(\xi)|^2\, d\xi \right) \left(\int_{\R^{2n}}|g(\xi)|^2\, d\xi \right).$$
\end{prop}
\begin{proof} Recall that $ \pi(\xi,\eta) = \pi_1(\xi,\eta,0) $ where $ \pi_\lambda $ are  the Schr\"odinger representations of $ \He^{2n} $ realised on $ L^2(\R^{2n}) .$ It is enough to show that 
\begin{equation}\label{matrix-coeff}
\langle f, \Pi_\lambda(\xi,\eta)g \rangle  = \langle f, \pi(\eta^\prime,-\xi^\prime)g \rangle,\,\, \, (\xi^\prime,\eta^\prime) = \Lambda (\xi,\eta)
\end{equation}
where the matrix $ \Lambda $ is defined  by
\begin{equation}\label{lambda-matrix} \Lambda = 
\begin{pmatrix}
     I_{2n}& \tanh(\lambda/2)\mathbf{J} \\
     (\lambda/2)\mathbf{J}& - (\lambda/2)\coth(\lambda/2)I_{2n} 
  \end{pmatrix}.
\end{equation}
In fact, once we have the above identity it follows that 
$$ \int_{\R^{2n}}\int_{\R^{2n}} |\langle f, \Pi_\lambda(\xi,\eta)g \rangle |^2\, d\xi\, d\eta =  \int_{\R^{2n}}\int_{\R^{2n}} |F\circ \Lambda(\xi,\eta)|^2\, d\xi\, d\eta $$
where $ F(\xi,\eta) = \langle f, \pi(\eta,-\xi)g \rangle.$ Since $ \mathbf{J}^2 = -I,$ an easy calculation shows that   the determinant of $ \Lambda $ is   $ (-\lambda/2)^{2n} \big( \coth (\lambda/2)-\tanh(\lambda/2)\big)^{2n} = \left( \frac{\lambda}{\sinh \lambda}\right)^{2n}$  and hence $ \Lambda $ is invertible. Therefore,
\begin{equation}
 \int_{\R^{2n}}\int_{\R^{2n}} |\langle f, \Pi_\lambda(\xi,\eta)g \rangle |^2\, d\xi\, d\eta =  |\det \Lambda|^{-1}\,  \int_{\R^{2n}}\int_{\R^{2n}} |F(\xi,\eta)|^2\, d\xi\, d\eta. 
 \end{equation}
But it is well known that the representation $ \pi_1 $ of $ \He^{2n} $ is square integrable modulo the center and 
$$\int_{\R^{2n}}\int_{\R^{2n}} |F(\xi,\eta)|^2\, d\xi\, d\eta =  (2\pi)^{2n} \left(\int_{\R^{2n}} |f(\xi)|^2\, d\xi \right) \left(\int_{\R^{2n}}|g(\xi)|^2\, d\xi \right).$$

Thus we are left with proving the identity \eqref{matrix-coeff}. 
With $ \xi =(x,u) $ and $ \eta =(y,v) $ let us compute the matrix coefficient
$$ \langle f, \Pi_\lambda(\xi,\eta)g\rangle=e^{-i\frac{\lambda}{2}(\coth \lambda)\xi\cdot\eta}\langle f, e_\lambda(\eta)\tau_\lambda(\xi)g\rangle.$$
Recalling the definitions of  $ \tau_\lambda(x,u)$ and $ e_\lambda(y,v) $ after a simple calculation we get
$$ e_\lambda(\eta)(\tau_\lambda(\xi)g)(a,b) = g(a-x^\prime,b-u^\prime)\, e^{i ( a \cdot y^\prime+ b\cdot v^\prime)}\, e^{-i\frac{\lambda}{2} \tanh (\lambda/2)( \xi \cdot \eta)}. $$
Here we have defined the new coordinates by  $ (x^\prime, u^\prime) = \big(x+\tanh (\lambda /2)\, v, u-\tanh (\lambda/2)\, y\big)$ and 
$ (y^\prime, v^\prime) = -(\lambda /2)\big(-u+\coth (\lambda/2)\,y, x+\coth (\lambda/2)\,v\big).$  If  we write $ \xi^\prime = (x^\prime, u^\prime), \eta^\prime = (y^\prime, v^\prime) ,$  then we see that  $ (\xi^\prime, \eta^\prime) = \Lambda(\xi,\eta) $ where the matrix $ \Lambda $ is as in the proposition.
We can now rewrite  the above expression for $ e_\lambda(\eta)\tau_\lambda(\xi)g $ in terms of ordinary translations and modulations and hence in terms of the representation $ \pi(\xi,\eta).$ By  making use of the relation 
$2\coth \lambda=\coth (\lambda/2)+\tanh (\lambda/2)$ and simplifying we get
$$  \Pi_\lambda(\xi,\eta)g(a,b) =   \pi(\eta^\prime,-\xi^\prime)g(a,b)  $$ 
and hence $  \langle f, \Pi_\lambda(\xi,\eta)g \rangle = \langle f, \pi(\eta^\prime,-\xi^\prime)g \rangle = F(\xi^\prime,\eta^\prime)$ proving the identity \eqref{matrix-coeff}.
\end{proof}

\subsection{Twisted modulation spaces}  Let $ L^{p,q}(\R^{2n} \times\R^{2n}) $ stand for the space of all measurable functions $ F(\xi,\eta)  $ on $ \R^{2n} \times \R^{2n} $ for which
$$ \|F\|_{p,q}^q  = \int_{\R^{2n}} \left( \int_{\R^{2n}} |F(\xi,\eta)|^p \,d\xi \right)^{q/p} d\eta $$
is finite.
We say that a tempered distribution  $ f $ on $ \R^{2n} $ belongs to the twisted modulation space $ M_\lambda^{p,q}(\R^{2n}) $ if the matrix coefficient 
$ \langle f, \Pi_\lambda(\xi,\eta)g \rangle $
 belongs to $ L^{p,q}(\R^{2n} \times\R^{2n}) $ for some Schwartz class function $ g.$ We equip  $ M_\lambda^{p,q}(\R^{2n} )$ with the norm 
$$ \|f\|_{(p,q)}^q  = \int_{\R^{2n}} \left( \int_{\R^{2n}} |\langle f, \Pi_\lambda(\xi,\eta)g \rangle|^p \,d\xi \right)^{q/p} d\eta.$$
In view of \eqref{matrix-coeff} we immediately see that $ M_\lambda^{p,p}(\R^{2n}) =M^{p,p}(\R^{2n})$  for any $ \lambda.$ But for $ p \neq q ,  M_\lambda^{p,q}(\R^{2n}) $ could be different from $ M^{p,q}(\R^{2n}).$
It follows from the definition of $ \Pi_\lambda $ that for $ f \in L^2(\R^{2n}) $ we have $$ \langle f, \Pi_\lambda(\xi,\eta)g \rangle = \langle B_\lambda f, \rho_\lambda(\xi+i\eta)B_\lambda g \rangle $$ and hence $ f \in M_\lambda^{p,q}(\R^{2n}) \cap L^2(\R^{2n}) $ if and only if $ \langle B_\lambda f, \rho_\lambda(\xi+i\eta)B_\lambda g \rangle $ belongs to $ L^{p,q}(\R^{2n} \times\R^{2n}) $ for some Schwartz class function $ g.$ This relation is very useful in studying certain properties of the twisted modulation spaces.

\begin{thm}\label{independence} The definition of $ M^{p,q}_\lambda(\R^{2n}) $ is independent of $ g .$ That is to say if $ \langle f, \Pi_\lambda(\xi,\eta)g \rangle $ belongs to $ L^{p,q}(\R^{2n} \times\R^{2n}) $ for some $ g $ then $ \langle f, \Pi_\lambda(\xi,\eta)g^\prime \rangle $ also belongs to $ L^{p,q}(\R^{2n} \times\R^{2n}) $ for any other $ g^\prime.$ Different windows $ g $ define equivalent norms.
\end{thm}

To prove this result, we need some preparation.  Let us define $V_g^\lambda f(\xi, \eta)=\langle f, e_\lambda(\eta)\tau_\lambda(\xi)g\rangle$ which is the twisted analogue of the short-time Fourier transform used in the classical theory of modulation spaces. In view of \eqref{pi-rep} it follows that $ f \in M^{p,q}_\lambda(\R^{2n}) $ if and only if  $ V_g^\lambda f \in  L^{p,q}(\R^{2n} \times\R^{2n}). $ Moreover, the identity in Proposition \ref{square-int} holds for $ V_g^\lambda f $ which after polarization gives 
\begin{equation}\label{orthogonality relation}
c_\lambda^{2} \,\langle V_g^\lambda f, V_{g^\prime}^\lambda h\rangle=   d_n\,  \langle f, h\rangle \langle g^\prime, g\rangle.
\end{equation}
We use the above identity  to obtain an inversion formula for $V_g^\lambda$. Towards this end, we define the following operator on $L^{p,q}(\R^{2n} \times\R^{2n}) $ which takes a function $ F $ into the functional
$$\langle \widetilde{V_g^\lambda}F, \varphi\rangle = \int_{\R^{2n}}\int_{\R^{2n}}F(\xi,\eta)\, \langle e_\lambda(\eta)\tau_\lambda(\xi) g, \varphi\rangle \,d\xi \, d\eta,$$
for  $\varphi\in \mathcal{S}(\R^{2n})$. In the following lemma we record some important properties of  $\widetilde{V_g^\lambda}$.

\begin{lem}\label{polarised-id} Let $1\leq p,q\leq \infty$ and $g, h \in \mathcal{S}(\R^{2n})$. Then
 $V_h^\lambda \circ \widetilde{V_g^\lambda}$ maps $L^{p,q}(\R^{2n} \times\R^{2n}) $ into itself and the following inequality holds:
\begin{equation}\label{norm-est}\| V_h^\lambda \circ \widetilde{V_g^\lambda}F\|_{p,q}\leq \|F\|_{p,q}\|V_g^\lambda h\|_{1,1}. \end{equation}
In particular, when $ F = V_g^\lambda f,$  we have the inversion formula
\begin{equation}\label{invers} 
c_\lambda^{2}\, \,\widetilde{V_{h}^\lambda}(V_g^\lambda f) =  d_n \,   \langle h, g\rangle \, f .\end{equation}
\end{lem}
\begin{proof} Since $ V_g^\lambda h $ can be written in terms of the classical short-time Fourier transform of $ g $ and $ h $ it is clear that $ V_g^\lambda h \in \mathcal{S}(\R^{2n}\times \R^{2n}).$ Note that $ |V_h^\lambda g(\xi,\eta)| =
| \langle g, \Pi_\lambda(\xi,\eta)h \rangle|$ and hence it is enough to consider 
$$ \langle \widetilde{V_g^\lambda} F, \Pi_\lambda(\xi,\eta)h \rangle =  \int_{\R^{2n}}\int_{\R^{2n}}F(\xi^\prime,\eta^\prime)\, \langle e_\lambda(\eta^\prime)\tau_\lambda(\xi^\prime) g, \Pi_\lambda(\xi,\eta)h\rangle \,d\xi^\prime \, d\eta^\prime.$$
Using  \eqref{pi-rep} and the fact that $ \Pi_\lambda(\xi,\eta)^\ast = \Pi_\lambda(-\xi,-\eta) $ we can rewrite the above as
$$  \int_{\R^{2n}}\int_{\R^{2n}}F(\xi^\prime,\eta^\prime) e^{i \frac{\lambda}{2} (\xi^\prime \cdot \eta^\prime)}\, \langle g, \Pi_\lambda(-\xi^\prime,-\eta^\prime)\Pi_\lambda(\xi,\eta)h\rangle \,d\xi^\prime \, d\eta^\prime.$$
As $ \Pi_\lambda $ is a projective representation of $ \R^{2n} \times \R^{2n} $ from the above we obtain
\begin{equation}\label{STFT pointise estimate} | V_h^\lambda \circ \widetilde{V_g^\lambda}F(\xi,\eta)| \leq  |F| \ast |V_g^\lambda h|(\xi,\eta).
\end{equation}
This yields \eqref{norm-est} by Young's inequality for mixed norm spaces.  In order to prove \eqref{invers} it is enough to show that
$$c_\lambda^{2} \,\, \langle \widetilde{V_{h}^\lambda}(V_g^\lambda f), \varphi \rangle  = d_n \,  \langle h, g\rangle \, \langle  f, \varphi \rangle  $$
for all $ \varphi \in \mathcal{S}(\R^{2n}).$ But this is certainly true for $ f \in L^2(\R^{2n}) $ as a consequence of \eqref{orthogonality relation}. Thus for  all $ \varphi \in  \mathcal{S}(\R^{2n})$ we have the inversion formula
$$  d_n \,c_\lambda^{-2}\,  \langle g, h\rangle \, \varphi =  \widetilde{V_{g}^\lambda}(V_h^\lambda \varphi)  = \int_{\R^{2n}}\int_{\R^{2n}} V_h^\lambda \varphi(\xi^\prime,\eta^\prime)\, 
e_\lambda(\eta^\prime)\tau_\lambda(\xi^\prime) g \,d\xi^\prime \, d\eta^\prime.$$ 
From the above by taking inner product with $ f $ and simplifying we get 
$$  d_n \,c_\lambda^{-2}\,  \langle h, g\rangle \, \langle f,\varphi \rangle  = \int_{\R^{2n}}\int_{\R^{2n}} V_g^\lambda f(\xi^\prime,\eta^\prime)\, 
\langle e_\lambda(\eta^\prime)\tau_\lambda(\xi^\prime) h, \varphi \rangle \,d\xi^\prime \, d\eta^\prime.$$
The change of order of integration is justified as $ g, h $ and $ \varphi $ are all Schwartz. As the right hand side of the above is nothing but $\langle \widetilde{V_{h}^\lambda}(V_g^\lambda f), \varphi \rangle$ \eqref{invers} is proved.
\end{proof}

{\bf Proof of Theorem \ref{independence}:} Assume that $ F = V_g^\lambda f \in L^{p,q}(\R^{2n} \times\R^{2n}) $ and take any Schwartz function $ g^\prime .$  By  \eqref{invers} we have $ d_n\, \|g\|_2^2 \, \,f = 
c_\lambda^2\,\, \widetilde{V_{g}^\lambda}(V_g^\lambda f)$ and hence using \eqref{norm-est} we get
$$ d_n\, \|g\|_2^2 \, \| V_{g^\prime}^\lambda f\|_{p,q} \leq  c_\lambda^2\,  \|V_g^\lambda f\|_{p,q}\|V_g^\lambda g^\prime \|_{1,1}. $$
This proves that $ V_{g^\prime}^\lambda f \in L^{p,q}(\R^{2n} \times\R^{2n}) $ and hence the definition of $ M_\lambda^{p,q}(\R^{2n}) $ independent of $ g.$ Moreover, it is clear that $ \| V_{g^\prime}^\lambda f\|_{p,q}$ is equivalent to $\| V_{g}^\lambda f\|_{p,q}.$\\

An interesting consequence of Theorem \ref{independence} is the following result which is a very useful alternate definition of $ M_\lambda^{p,q}(\R^{2n}) .$ Let  $ w_\lambda(\xi+i\eta) $ which appears in the definition of $ \Fs.$ \\

\begin{cor} \label{matrix-bargmann}A function $ f \in M_\lambda^{p,q}(\R^{2n}) \cap L^2(\R^{2n})  $ if and only if  the function  $ F $ defined by
$  F(\xi,\eta) = \sqrt{w_\lambda(\xi+i\eta)}\, \, B_\lambda f(\xi+i\eta)  $ belongs to $L^{p,q}(\R^{2n} \times\R^{2n}) .$
\end{cor} 
\begin{proof} In view of the above theorem, $ f \in M_\lambda^{p,q}(\R^{2n}) $ if and only if $ \langle f, \Pi_\lambda(\xi,\eta) p_{1/2}^\lambda \rangle $ belongs to  $L^{p,q}(\R^{2n} \times\R^{2n}) .$ As $ B_\lambda $ is unitary, $ B_\lambda \,p_{1/2}^\lambda = c_\lambda $  and $ \Pi_\lambda(\xi,\eta) = B_\lambda^\ast \circ \rho_\lambda(\zeta)\circ B_\lambda,\,  $ we have $ \langle f, \Pi_\lambda(\xi,\eta) p_{1/2}^\lambda \rangle = c_\lambda\,  \langle B_\lambda f, \rho_\lambda(\zeta)1 \rangle $  where $\zeta = \xi+i\eta.$  From the definition of $ \rho_\lambda(\zeta) $ given in \eqref{fock-rep} we see that $\langle B_\lambda f, \rho_\lambda(\zeta)1 \rangle $ is given by the integral 
$$  e^{-\frac{\lambda}{4} (\coth \lambda)|\zeta|^2}\,  e^{\frac{\lambda}{2} [ \Re \zeta,\,\Im \zeta ]}\,\, \int_{\C^{2n}} B_\lambda f(\zeta^\prime) e^{ \frac{\lambda}{2} (\coth \lambda)\, (\zeta \cdot \overline{\zeta^\prime})}\, e^{-i\frac{\lambda}{2}\,  [\zeta, \overline{\zeta^\prime}]}\, w_\lambda(\zeta^\prime)\, d\zeta^\prime
.$$
Recalling  the expression for the reproducing kernel  for $ \Fs $ given in  \eqref{rep-ker}, namely
$$ d_n\, c_\lambda\,\, e^{ \frac{\lambda}{2} (\coth \lambda)\, (\zeta \cdot \overline{\zeta^\prime})}\, e^{-i\frac{\lambda}{2}\,  [\zeta, \overline{\zeta^\prime}]} =  \, \overline{K_\zeta(\zeta^\prime}),$$ 
the above integral is seen to be just $ B_\lambda f(\zeta) .$  Therefore,
 $$  d_n\, c_\lambda\,\, \langle B_\lambda f, \rho_\lambda(\zeta)1 \rangle = \,  B_\lambda f(\zeta)\, e^{-\frac{\lambda}{4} (\coth \lambda)|\zeta|^2}\,  e^{\frac{\lambda}{2} [ \Re \zeta,\,\Im \zeta ]}.$$ 
 Recalling the definition of the weight function $ w_\lambda $ we can rewrite the above as 
$$ d_n\,\, c_\lambda\, \sqrt{c_\lambda}\, \langle B_\lambda f, \rho_\lambda(\zeta)1 \rangle =  B_\lambda f(\zeta)\,\sqrt{w_\lambda(\zeta)}.$$
This proves the corollary.
\end{proof}

\begin{rem} We record  the following relation established above for future use:
\begin{equation}\label{matrix-barg}  d_n\, \sqrt{c_\lambda} \, \langle f, \Pi_\lambda(\xi,\eta) p_{1/2}^\lambda \rangle =  B_\lambda f(\zeta)\,\sqrt{w_\lambda(\zeta)} .
\end{equation}   
From the above relation we immediately obtain the identity
$$  c_\lambda \, \int_{\R^{2n}}\int_{\R^{2n}}  |\langle f, \Pi_\lambda(\xi,\eta) p_{1/2}^\lambda \rangle|^2\,d\xi\, d\eta = \int_{\R^{2n}} |f(x,u)|^2 dx \, du$$
which plays an important role in the definition of modulation spaces on the Heisenberg group.
\end{rem}

In the following theorem we list  some basic properties of the twisted modulation spaces.

\begin{thm}\label{lambda Msp prop} (i) The space  $ M_\lambda^{p,q}(\R^{2n})$ is a Banach space for any $ 1 \leq p, q \leq \infty.$ (ii) When $ p$ and $ q $ are finite the Schwartz space is dense in  $ M_\lambda^{p,q}(\R^{2n})$ and the dual is given by $ M_\lambda^{p^\prime,q^\prime}(\R^{2n})$ (iii) The inclusion
    $$M_\lambda^{p_1,q_1}(\R^{2n})\subseteq M_\lambda^{p_2,q_2}(\R^{2n}).$$
holds whenever  $1\leq p_1\leq p_2\leq\infty$, $1\leq q_1\leq q_2\leq\infty$. 
\end{thm}
\begin{proof} The proofs are very similar to the classical case. For example, to  see the completeness of $  M_\lambda^{p,q}(\R^{2n})$  let $ (f_k) $ be a Cauchy sequence. Then by definition, $ F_k = V_g^\lambda f_k $ is Cauchy in  $L^{p,q}(\R^{2n} \times\R^{2n}) $ where $ g $ is a fixed Schwartz function 
with $ \|g\|_2 =1.$ By the completeness of  $L^{p,q}(\R^{2n} \times\R^{2n}),$ we get $ F \in  L^{p,q}(\R^{2n} \times\R^{2n}) $ to which the sequence $ (F_k) $ converges. In view of \eqref{invers} we have $ c_\lambda\,  f_k =  \widetilde{V} _g^\lambda F_k $ and by defining $ c_\lambda f = \widetilde{V} _g^\lambda F $  we see that 
$$ c_\lambda\,  \| V_g^\lambda (f_k-f) \|_{p,q} = c_\lambda\,  \| V_g^\lambda \circ \widetilde{V}_g^\lambda (F_k-F) \|_{p,q} \leq C\, \|F_k -F\|_{p,q} $$
in view of the estimate \eqref{norm-est}. This proves that $ (f_k) $ converges to $ f $ in $  M_\lambda^{p,q}(\R^{2n}).$ The duality is proved as in the classical case. We refer to \cite{KG} for the details. To prove  the density of the Schwartz space, we require the easy to prove fact that $ \widetilde{V}_g^\lambda F $ is Schwartz whenever $ F $ is compactly supported or rapidly decreasing. In order to see the inclusion, take $f\in M^{p_1,q_1}_\lambda(\R^{2n})$. Then using the inversion formula \eqref{invers} and the estimate \eqref{STFT pointise estimate}, we get
$$ |V_g^\lambda f(\xi,\eta)|\leq c_n\,c_\lambda^2\,\langle g, g\rangle^{-1} |V_g^\lambda f|\ast|V_g^\lambda g|(\xi,\eta). $$
Now choose $ 1\leq r,s\leq \infty$ such that
$$ 1/p_1+1/r=1+1/p_2, ~~\text{ and } 1/q_1+1/s=1+1/q_2. $$
By Young's inequality for mixed Lebesgue spaces we obtain
\begin{equation}\label{lambda inclusion estimate}
 \|V_g^\lambda f\|_{p_2,q_2}\leq c_n\, c_\lambda^2\,\langle g, g\rangle^{-1} \|V_g^\lambda f\|_{p_1,q_1}\|V_g^\lambda g\|_{r,s}.
\end{equation}
Since $ g$ is a Schwartz function, $ V_g^\lambda g\in \mathcal{S}(\R^{2n} \times \R^{2n})$  and hence  the desired inclusion follows from the above estimate.  
\end{proof}

In the classical setting, the spaces $ M^{p,q}(\R^n) $ are invariant under translations and modulations. Moreover, when $ p = q $ they are also invariant under Fourier transform. In the twisted setting we have the following invariance properties.

\begin{thm} The spaces $  M_\lambda^{p,q}(\R^{2n})$ are invariant under $ \tau_\lambda(\xi) $ and $ e_\lambda(\eta) $ for all $ \xi, \eta \in \R^{2n}.$ When $ p=q$  they are invariant under $ U_\lambda $ defined in \eqref{intertwine}.
\end{thm} 
\begin{proof} The invariance under $ \tau_\lambda(\xi) $ and $ e_\lambda(\eta) $ follows from the definition.  From the definition of $ \rho_\lambda(\zeta) $ we easily calculate that $ U \circ \rho_\lambda(\zeta) \circ U^\ast = \rho_\lambda (i\zeta) $  from which follows 
$$  U \circ B_\lambda \circ \Pi_\lambda(\xi,\eta) \circ B_\lambda^\ast \circ U^\ast = \rho_\lambda(i\zeta).$$ 
The property $ U \circ B_\lambda  = B_\lambda \circ U_\lambda $ translates the above into 
$$ U_\lambda \circ \Pi_\lambda(\xi,\eta)  \circ  U_\lambda^\ast =  B_\lambda^\ast \circ \rho_\lambda(i\zeta) \circ B_\lambda = \Pi_\lambda(-\eta, \xi).$$
If $ f \in  M_\lambda^{p,p}(\R^{2n})$  so that $ \langle f, \Pi_\lambda(\xi,\eta)g \rangle \in L^{p,p}(\R^{2n} \times \R^{2n}) $ then the equation
$$ \langle f, \Pi_\lambda(\xi,\eta)g \rangle = \langle f,  U_\lambda^\ast \circ \Pi_\lambda(-\eta,\xi) \circ U_\lambda g\rangle = \langle U_\lambda f, \Pi_\lambda(-\eta,\xi) U_\lambda g\rangle $$
shows that $ U_\lambda f \in  M_\lambda^{p,p}(\R^{2n})$ proving the invariance of  $ M_\lambda^{p,p}(\R^{2n})$ under $ U_\lambda.$
\end{proof}

\begin{rem} It is well known that the modulation space $ M^{1,1}(\R^{2n}) = M^{1,1}_\lambda(\R^{2n})$ is a Banach algebra under pointwise multiplication. As $ M^{1,1}$ is invariant under the Fourier transform, $ M^{1,1}(\R^{2n})$ is a Banach algebra under convolution as well. We can also show that $ M^{1,1}(\R^{2n}) $ is a Banach algebra under $ \lambda$-twisted convolution. To see this, recall  \eqref{t-con} which reads as
$$ f \ast_\lambda g(\xi) = \int_{\R^{2n}} f(\eta)\,  \tau_\lambda(\eta)g(\xi)\, d\eta .$$
Since $ M^{1,1}(\R^{2n}) $ is invariant under $ \tau_\lambda(\eta)$   by Minkowski's integral  inequality we have 
$$ \| f \ast_\lambda g \|_{(1,1)} \leq \|f \|_1\, \| g\|_{(1,1)} \leq \|f\|_{(1,1)}\, \|g\|_{(1,1)} .$$
In the last inequality we have used the continuous  the inclusion $ M^{1,1}(\R^{2n}) \subset L^1(\R^{2n}).$  
\end{rem}

The same property holds for $ M_\lambda^{p,p}(\R^{2n})$ as long as $ 1 \leq p \leq 2 $ as show below. The proof makes use of the curious property enjoyed by the twisted convolution, namely 
\begin{equation}\label{curious} \| f \ast_\lambda g \|_p \leq \|f\|_p \, \|g\|_p \, \end{equation}
valid for $ 1 \leq p \leq 2.$ The above is clearly true for $ p = 1$ and for $ p =2 $ it is a consequence of the fact that Weyl transform is an isometry from $ L^2(\R^{2n}) $ onto the space of Hilbert-Schmidt operators. For $ 1 \leq p \leq 2 $ it is proved by interpolation, see  Proposition 3.4 in \cite{GMa}  and Theorem 1.1 in \cite{PKR} for details.\\

\begin{thm} For any $ 1 \leq p \leq 2 $ the space $ M_\lambda^{p,p}(\R^{2n}) = M^{p,p}(\R^{2n})$ is a Banach algebra under twisted convolution.
\end{thm} 
\begin{proof} Assuming  that $ f, g \in M_\lambda^{p,p}(\R^{2n}),$ with $ h \in \mathcal{S}(\R^{2n}) $  consider 
$$ \langle f \ast_\lambda g,  \Pi_\lambda(\xi,\eta)h \rangle = \int_{\R^{2n}} f(\xi^\prime)\,  \langle \tau_\lambda(\xi^\prime)g, \Pi_\lambda(\xi,\eta)h \rangle\, d\xi^\prime .$$
Since $ \tau_\lambda(\xi^\prime) = \Pi_\lambda(\xi^\prime,0) $ a simple calculation using \eqref{Pi_lambda composition formula intro} and the fact that $ \Pi_\lambda $ is a representation of $ \He^n_\lambda(\C)$  gives us
$$ \tau_\lambda(-\xi^\prime) \Pi_\lambda(\xi,\eta) = \Pi_\lambda(\xi-\xi^\prime,\eta)\, e^{i \frac{\lambda}{2} (\coth \lambda) \xi^\prime \cdot \eta} \, e^{-i\frac{\lambda}{2} [\xi,\xi^\prime]}.$$
Therefore, by defining $ f_\eta(\xi^\prime) = f(\xi^\prime) \,e^{i \frac{\lambda}{2} (\coth \lambda) \xi^\prime \cdot \eta} $ and $ G_\eta(\xi^\prime) = \langle g, \Pi_\lambda(\xi^\prime,\eta)h \rangle\ $ we can write
$$ \langle f \ast_\lambda g,  \Pi_\lambda(\xi,\eta)h \rangle = f_\eta \ast_\lambda G_\eta(\xi).$$
We now make use of  \eqref{curious} to conclude that 
$$ \| \langle f \ast_\lambda g,  \Pi_\lambda(\cdot,\eta)h \rangle \|_p \leq \| f_\eta \|_p \, \| G_\eta\|_p.$$
Since $ \| f_\eta \|_p = \|f\|_p \leq \|f\|_{(p,p)} $ integrating  the $ p$-th power of the above with respect to $ \eta $ we complete the proof.
\end{proof}

\section{Modulation spaces on the Heisenberg group}

\subsection{ Translations and modulations on the Heisenberg group} The family of unitary operators $ \Pi_\lambda(\xi,\eta) $ 
satisfy the composition formula
$$  \Pi_\lambda (\zeta)\, \Pi_\lambda(\zeta^\prime) =  \Pi_\lambda(\zeta+\zeta^\prime) e^{-i \frac{\lambda}{2} ((\coth \lambda)\, \Im (\zeta \cdot \overline{\zeta^\prime})- \Re [\zeta, \overline{\zeta^\prime}])}.$$
This suggests that we introduce a group structure on $ \C^{2n+1} = \C^{2n} \times \C $ by the rule
$$ (\zeta, s) (\zeta^\prime,s^\prime) =  (\zeta+\zeta^\prime, s+s^\prime +\frac{1}{2}\Re [\zeta, \overline{\zeta^\prime}]+\frac{i}{2} \Im (\zeta \cdot \overline{\zeta^\prime})) .$$
We denote this group by $ G_n $ and note that $ \R^{2n} \times \R $ and $ i\R^{2n} \times \R $ are subgroups of  $ G_n $ each isomorphic to $ \He^n.$  If we let
$$  \Pi_\lambda (\zeta, s) = e^{i \lambda( \Re s -(\coth \lambda)\Im s)} \,  \Pi_\lambda (\zeta),$$
we easily verify that $ \Pi_\lambda(g) \Pi_\lambda(g^\prime) = \Pi_\lambda(g g^\prime)$ for all $ g, g^\prime \in G_n.$ Thus we have a family of unitary representations  of $ G_n $ on $ L^2(\R^{2n}) .$ These are also irreducible as the family $ \rho_\lambda(\zeta) $ is irreducible as shown in \cite{GT}.  The operators $ \Pi_\lambda(\zeta,s) $ initially defined on $ L^2(\R^{2n})$ can  be lifted to unitary operators acting on $ L^2(\He^n) $ when $ s $ is real.  We simply define
\begin{equation}\label{Pi-def} \Pi(\zeta,s)f(x,u,t) = (2\pi)^{-1} \int_{-\infty}^\infty e^{-i\lambda t}\,  \, \Pi_\lambda(\zeta,s)f^\lambda(x,u) \, d\lambda.\\
\end{equation}

As $ \Pi_\lambda(\zeta,s) $ acts on $ L^2(\R^{2n}) $ as unitary operators, it is clear that $ \Pi(\zeta,s) $ is an isometry on $ L^2(\He^n).$ An easy calculation shows that $ \Pi(\zeta,s)^\ast = \Pi(-\zeta,-s) $
and hence $ \Pi(\zeta,s) $ is in fact unitary. Moreover, for $ s \in \R $ and  $ \xi  \in \R^{2n}$ it follows that
$$ \Pi(\xi,s)f(\eta,t) =(2\pi)^{-1} \int_{-\infty}^\infty e^{-i\lambda t}\, e^{i\lambda s} \, \tau_\lambda(\xi)f^\lambda(\eta) \, d\lambda = \tau(\xi,s)f(\eta,t)$$
where $ \tau(g)f(h) = f(g^{-1}h), \, g, h \in \He^n $ is the left translation on the Heisenberg group. The family of operators $U_\lambda $ also gives rise to the operator
$$ \widetilde{U}f(\xi,t) = (2\pi)^{-1} \int_{-\infty}^\infty e^{-i\lambda t}\,  \, U_\lambda f^\lambda(\xi) \, d\lambda $$
which is unitary on $ L^2(\He^n).$  For $ s \in \R $ and $ \eta  \in \R^{2n} $ let us define
$$ e(\eta,s)f(\xi,t) = \Pi(i\eta,s)f(\xi,t) =(2\pi)^{-1} \int_{-\infty}^\infty e^{-i\lambda t}\, e^{i\lambda s} \, e_\lambda(\eta)f^\lambda(\xi) \, d\lambda.$$
As verified in \eqref{trans-mod},  $ U_\lambda \circ \tau_\lambda(a,b) \circ U_\lambda^\ast = e_\lambda(a,b),$  and hence it follows that $  \widetilde{U} \circ \tau(g) \circ  \widetilde{U}^\ast = e(g) $ for any $ g \in \He^n.$ Thus we may consider $ e(g) $ as modulation operators on the Heisenberg group.\\

\begin{rem}\label{Schwarz modulation Hsn} When $ f $ is a Schwartz function on $ \He^n$ it is clear that $ \tau(g)f $ is also Schwartz for any $ g \in \He^n.$  We are interested in knowing if the  same thing is true for $ e(g)f.$ In view of the relation 
$  \widetilde{U}\circ \tau(g) \circ \widetilde{U}^\ast = e(g), $  this property holds if $ \widetilde{U} $ leaves the Schwartz class invariant. In the following proposition we show that this is indeed true.
\end{rem}

\begin{prop} For any Schwartz function $ f $ on $ \He^n, \, \widetilde{U} f$ is also Schwartz.
\end{prop}
\begin{proof} Recalling the definition of $ \widetilde{U}, $ and $ U_\lambda$ we are required to show that 
$$  U_\lambda f^\lambda(\xi) = c(\lambda)^n\, \mathcal{F}_{2n}f^\lambda(c(\lambda)\xi) $$ is a Schwartz function of $ (\xi,\lambda) $ on $ \R^{2n} \times \R.$ Here $ \mathcal{F}_{2n} $ is the Fourier transform on $ \R^{2n}.$ It is therefore enough to consider $ F(\xi,\lambda) = c(\lambda)^n\, g(c(\lambda)\xi,\lambda) $ when $ g $ is Schwartz. We will show that for any multi-indices $ \alpha $ and $ \beta $ and non-negative integers $ j $ and $ p$ the functions $ F_{\alpha,\beta,j,p}(\xi,\lambda) = \lambda^p  \partial_\lambda^j\, \xi^\beta \partial_\xi^\alpha F(\xi,\lambda) $ are bounded on $ \R^{2n}\times \R.$ Note that 
$$ \partial_\xi^\alpha g(c(\lambda)\xi,\lambda)  = c(\lambda)^{|\alpha|}\,  \, (\partial_\xi^\alpha g)(c(\lambda)\xi,\lambda),$$
$$  \partial_\lambda \,g(c(\lambda)\xi,\lambda)  = (\partial_\lambda g)(c(\lambda)\xi,\lambda) + \frac{d}{d\lambda}c(\lambda)\,  \, \sum_{j=1}^{2n} \xi_j (\partial_{\xi_j}g)(c(\lambda)\xi,\lambda)$$
and therefore a moment's thought reveals that $ \partial_\lambda^j\,  \partial_\xi^\alpha F(\xi,\lambda) $ is a finite linear combination of terms of the form  
$$ c(\lambda)^l \, \frac{d^m}{d\lambda^m} c(\lambda)\, \xi^\mu \,(\partial_\xi^\nu \partial_\lambda^kg)(c(\lambda)\xi,\lambda).$$
The boundedness of $ F_{\alpha,\beta,j,p}(\xi,\lambda)$ follows immediately once we have the following estimates on $ c(\lambda) $ and its derivatives.
\end{proof}

\begin{lem}\label{clambda} For any  $ \lambda \in \R $ and a non-negative integer $ m $ we have $ |c(\lambda)| \geq c_1 $ and 
$$ |\frac{d^m}{d\lambda^m} c(\lambda)| \leq c_2 (1+|\lambda|)^{m+1}.$$
\end{lem}
\begin{proof}  As $ c(\lambda) = \lambda \coth \lambda $ it is clear that the lemma is true for $ m =0.$ By defining $ \varphi(\lambda) = \left(\frac{\lambda}{\sinh \lambda}\right) $ and $ \psi(\lambda)= \varphi(\lambda)^{-1} $ we note that 
$$  \psi^{(m)}(\lambda) = \frac{1}{2} \int_{-1}^1 t^m\,e^{\lambda t} dt $$
which shows that  $ \psi^{(m)}(\lambda)$ grows like $ \cosh \lambda.$ From the equation $ c(\lambda) \psi(\lambda) = \cosh \lambda$ we get
$$ c^{(m)}(\lambda)\, \psi(\lambda) = \cosh \lambda  - \sum_{j=0}^{m-1}  \frac{m!}{j! (m-j)!} \,c^{(j)}(\lambda) \, \psi^{(m-j)}(\lambda)$$
when $ m $ is even. (For odd $ m $ the above formula holds after changing  $ \cosh \lambda $ into $ \sinh \lambda$ on the right.) By induction we get the estimates on the derivatives of $ c(\lambda).$  The lower bound for $ c(\lambda) $ follows from $ c(\lambda)^{-1} = \psi(\lambda) (\cosh \lambda)^{-1}.$
\end{proof}

\subsection{Matrix coefficients associated to the representation $ \Pi$} Recall that the modulation spaces $ M^{p,q}(\R^n)$ are defined in terms of the matrix coefficients $ \langle f, \pi_1(x,y,t)g \rangle $ associated to the Schr\"odinger representation of the Heiseberg group. In particular, the identification $ M^{2,2}(\R^n) = L^2(\R^n) $ is a consequence of the fact that $ \langle f, \pi_1(x,y,t)g \rangle $ is square integrable modulo the centre of  $ \He^n $ and the identity
\begin{equation}\label{fourier-wigner} \int_{\R^{2n}} |\langle f, \pi_1(x,y,t)g \rangle|^2\, dx\, dy\, = c_n \left ( \int_{\R^n} |f(x)|^2 \, dx \right) \left ( \int_{\R^n} |g(x)|^2 \, dx \right).
\end{equation}
We now consider the matrix coefficients $ \langle f, \Pi(\zeta,s) g \rangle $ of the representation $ \Pi$ of $ G_n $ where $ f, g \in L^2(\He^n)$  and see if an identity of the above form is true. For $ \zeta = \xi+i\eta $ and $ s \in \R,$
$$ \langle f, \Pi(\zeta,s) g \rangle = c_n  \int_{-\infty}^\infty e^{-i\lambda s}\, \langle  f^\lambda, \Pi_\lambda(\xi,\eta) g^\lambda \rangle \, d\lambda. $$
By Plancherel theorem for the Fourier transform in the $ s$ variable we get
$$ \int_{\C^{2n}} \int_{-\infty}^\infty |\langle f, \Pi(\zeta,s) g \rangle|^2\, d\zeta\, ds  = c_n  \int_{-\infty}^\infty \left(\int_{\R^{4n}} | \langle  f^\lambda, \Pi_\lambda(\xi,\eta) g^\lambda \rangle|^2  \, 
d\xi\,d\eta \right) d\lambda. $$
The inner integral has been already calculated in Proposition \ref{square-int}. Using that we obtain
\begin{equation}\label{norm-Vgf} \int_{\C^{2n}} \int_{-\infty}^\infty |\langle f, \Pi(\zeta,s) g \rangle|^2\, d\zeta\, ds  = c_n  \int_{-\infty}^\infty \left(\frac{\sinh \lambda}{\lambda}\right)^{2n}  \|f^\lambda\|_2^2\, \|g^\lambda\|_2^2\, d\lambda.
\end{equation}
Thus we see that the representation $ \Pi $ of $G_n $ fails to be square integrable.\\

Since the modulation spaces $ M^{p,q}(\R^n)$ are also defined in terms of  the Bargmann transform which corresponds to $ \langle f, \pi_1(x,y,t)q_{1/2} \rangle $ where $ q_t $ is the euclidean heat kernel, let us try to see the behaviour of $ \langle f, \Pi(\zeta,s) p_{1/2} \rangle $ where $ p_t $ is the heat kernel associated to the sublaplacian on $ \He^n.$ Recall the relation  \eqref{matrix-barg}, namely
$$ d_n\, \sqrt{c_\lambda} \, \langle f^\lambda, \Pi_\lambda(\xi,\eta) p_{1/2}^\lambda \rangle =  B_\lambda f^\lambda(\zeta)\,\sqrt{w_\lambda(\zeta)} $$   
established in the course of the proof of Corollary \ref{matrix-bargmann}.  We also have
\begin{equation}\label{mat-co} \langle f, \Pi(\zeta,s) p_{1/2} \rangle = c_n  \int_{-\infty}^\infty e^{-i\lambda s}\, \langle  f^\lambda, \Pi_\lambda(\xi,\eta) p_{1/2}^\lambda \rangle \, d\lambda .
\end{equation}
 Let $ T $ stand for the Fourier multiplier operator corresponding to $ \sqrt{c_\lambda} $ acting on $ f $ in the central variable. Then we have
\begin{equation}\label{matrix-modi} \langle Tf, \Pi(\zeta,s) p_{1/2} \rangle = c_n  \int_{-\infty}^\infty e^{-i\lambda s}\, B_\lambda f^\lambda(\zeta)\,\sqrt{w_\lambda(\zeta)}\, \, d\lambda. 
\end{equation}
As $ B_\lambda $ is an isometry we obtain the identity
\begin{equation}\label{matrix-planch} \int_{\C^{2n}} \int_{-\infty}^\infty |\langle Tf, \Pi(\zeta,s) p_{1/2} \rangle|^2\, d\zeta\, ds  =  c_n\, \int_{\He^n} |f(h)|^2 \,dh.
\end{equation}
Note  that $\langle Tf, \Pi(\zeta,s) p_{1/2} \rangle = T\langle f, \Pi(\zeta,\cdot) p_{1/2} \rangle(s) $ and hence the modified matrix coefficients $T\langle f, \Pi(\zeta,\cdot) p_{1/2} \rangle(s) $ are square integrable over $ \C^{2n}\times \R.$\\

 Something interesting happens if we replace the heat kernel $ p_t$ associated to the sublaplacian by the heat kernel $ h_t $ for the full Laplacian $ \Delta_{\He}$  on $ \He^n.$  Since $ \Delta_{\He} = -\frac{\partial^2}{\partial t^2} +\mathcal{L}, $ it follows that $ h_t^\lambda(x,u) = e^{-\frac{1}{2}\lambda^2} \, p_t^\lambda(x,u) $ and hence
\begin{equation}\label{heisen-segal-barg}
 \langle Tf, \Pi(\zeta,s) h_{1/2} \rangle = c_n  \int_{-\infty}^\infty e^{-i\lambda s}\, e^{-\frac{1}{2}\lambda^2}\, B_\lambda f^\lambda(\zeta)\,\sqrt{w_\lambda(\zeta)}\, \, d\lambda. 
\end{equation}
From the above expression it is clear that the function $ s \rightarrow  \langle Tf, \Pi(\zeta,s) h_{1/2} \rangle $ has a holomorphic extension to the whole of $ \C.$  By defining
\begin{equation}\label{heisen-barg}
B_{\He}f(\zeta,s) = e^{\frac{1}{4}s^2}\,  \langle Tf, \Pi(\zeta,s) h_{1/2} \rangle 
\end{equation}
we observe that $ B_{\He}f(\zeta,s) $ is nothing but the classical Bargmann transform of the function $ \lambda \rightarrow B_\lambda f^\lambda(\zeta)\,\sqrt{w_\lambda(\zeta)}$ evaluated at $-is.$ 
Consequently, we have the identity 
\begin{equation}\label{barg-identity}
 \int_{\C} |B_{\He}f(\zeta, s)|^2\,  e^{-\frac{1}{2}|s|^2}\,  ds = c_n \int_{-\infty}^\infty |B_\lambda f^\lambda (\zeta)|^2 \, w_\lambda(\zeta)\, d\lambda.
 \end{equation}
 Since the image of $ L^2(\R) $ under the Bargmann transform is a reproducing kernel Hilbert space we also have the pointwise estimate
 \begin{equation}\label{esti-point}
 |B_{\He}f(\zeta,s)|^2 \leq c_n\, e^{\frac{1}{2}|s|^2}\, \int_{-\infty}^\infty |B_\lambda f^\lambda (\zeta)|^2 \, w_\lambda(\zeta)\, d\lambda.
\end{equation}
In view of \eqref{matrix-modi} the above gives the following estimate valid for $ s \in \R$:
 \begin{equation}\label{esti-point-one}
 |B_{\He}f(\zeta,s)|^2 \leq c_n\, e^{\frac{1}{2}|s|^2}\, \int_{-\infty}^\infty  |\langle Tf, \Pi(\zeta,s) p_{1/2} \rangle|^2\, ds.
\end{equation}
This inequality and the identity \eqref{matrix-planch} motivate our definition of modulation spaces on the Heisenberg group.\\

We would like to extend the definition of the matrix coefficients $\langle f, \Pi(\zeta,s) g \rangle$ when $ f $ is a tempered distribution and $ g $ a Schwartz class function. This can be done once we know that $ \Pi(\zeta,s)g $ is Schwartz  whenever $ g $ is.

\begin{lem}\label{schwartz}  For any $ g \in \mathcal{S}(\He^n) $ and $ (\zeta,s) \in \C^{2n} \times \R ,$ the function $ \Pi(\zeta,s)g \in \mathcal{S}(\He^n).$
\end{lem}
\begin{proof} We have already noted in Remark \ref{Schwarz modulation Hsn} that both $ \Pi(\xi,s) $ and $ \Pi(i\eta,s) $ for $ \xi, \eta \in \R^{2n}$ preserve the Schwartz class as they are translations and modulations.  In view of the definition \eqref{Pi-def} and the relation \eqref{pi-rep} we have
$$\Pi(\zeta,s)g(x,u,t) = (2\pi)^{-1} \int_{-\infty}^\infty e^{-i\lambda t}\, e^{i\lambda s} \, e^{i\frac{\lambda}{2} (\coth \lambda)\, \xi \cdot \eta}\, e_\lambda(\eta) \, \tau_\lambda(\xi)\,g^\lambda(x,u) \, d\lambda.$$
It is therefore enough to show that the function
$$ g_0(x,u,t) =(2\pi)^{-1} \int_{-\infty}^\infty e^{-i\lambda t}\, e^{i\lambda s} \, e^{i\frac{\lambda}{2} (\coth \lambda)\, \xi \cdot \eta}\, g^\lambda(x,u) \, d\lambda$$
is Schwartz. But this is immediate in view of the properties of $ c(\lambda) = \lambda/2\,  \coth \lambda/2$ proved in Lemma \ref{clambda}. Note that it is important to assume $ s $ is real.
\end{proof}

\subsection{Modulation spaces on the Heisenberg group}

  We look for Banach spaces of tempered distributions on $ \He^n $ which are  invariant under translations $ \tau(h) $ and modulations $ e(h)$ for all $ h \in \He^n$ or equivalently invariant under the operators $ \Pi(\zeta,s), \zeta \in \C^{2n}, s \in \R.$ We look for a definition of $ M^{p,q}(\He^n) $ which  reduces to $ L^2(\He^n) $ when $ p =q=2.$  We also like to have the  Schwartz space $ \mathcal{S}(\He^n) $ contained in $ M^{p,q}(\He^n) .$ 
  We therefore, begin by investigating the behaviour of  $   \langle Tf, \Pi(\zeta,s) p_{1/2} \rangle $ when $ f $ is a Schwartz class function.\\

Let us try to get some pointwise estimates for $ B_\lambda f^\lambda(\zeta) $ when $ f $ is Schwartz.  Since the reproducing kernel for $ \Fs $ is given by  \eqref{rep-ker} it follows that 
$$ | B_\lambda f^\lambda(\zeta)|^2 \leq \| B_\lambda f^\lambda \|_{\Fs}^2 \,  K_\zeta(\zeta) = c_\lambda^2\, \left(\int_{\R^{2n}} |f^\lambda(x,u)|^2 \, dx\, du \right) w_{\lambda}^{-1}.$$
Therefore, we get the pointwise estimate
    $$ |B_\lambda f^\lambda(\zeta)| \sqrt{w_\lambda(\zeta)} \leq   c_\lambda \,\left(\int_{\R^{2n}} |f^\lambda(x,u)|^2 \, dx\, du \right)^{1/2} \leq   c_\lambda\, \int_{-\infty}^\infty \|f(\cdot,t)\|_2 \, dt.$$
Defining $ f_m(x,u,t) = (1-\partial_t^2)^{m} f(x,u,t) $ and noting that $ (1+\lambda^2)^m\, f^\lambda(x,u) = f_m^\lambda(x,u) $ we can improve the above estimate as
\begin{equation}\label{pointwise-barg-one} |B_\lambda f^\lambda(\zeta)| \sqrt{w_\lambda(\zeta)} \leq  c_\lambda\, ( 1+\lambda^2)^{-m}\,  \int_{-\infty}^\infty \|f_m(\cdot,t)\|_2 \, dt.
\end{equation}
We  would also like to get  an estimate of $B_\lambda f^\lambda(\zeta) \sqrt{w_\lambda(\zeta)} $ involving decay in the $\zeta$ variable. Such an estimate has been proved in \cite{RRST} when $ \lambda =1.$ As we need an estimate involving both $ \zeta $ and $ \lambda $ variables, we have to keep track of the constants  that depend on  $ \lambda.$\\

\begin{lem} Let $ a = \lambda/2\, \coth (\lambda/2)$ and $ b = i \lambda/2 $ and define $ A_j = \big( \frac{\partial}{\partial x_j} - a x_j \big) $and   $B_j = \big( - \frac{\partial}{\partial u_j} + au_j \big) $  for $ j = 1,2,...,n.$ Then we have the following relations:
$$ -(a^2+b^2) z_j \, B_\lambda f^\lambda(z,w) =  B_\lambda \big(a A_j- b B_j \big)f^\lambda(z,w),$$
$$ (a^2+b^2) w_j \, B_\lambda f^\lambda(z,w) =  B_\lambda \big(b A_j+a B_j \big)f^\lambda(z,w).$$
\end{lem}
\begin{proof} By setting $ e^{-\frac{1}{2}L_\lambda}f^\lambda(z,w) = f^\lambda \ast_\lambda p_{1/2}^\lambda(z,w) $ we have
$$ e^{-\frac{1}{2}L_\lambda} (\frac{\partial}{\partial x_j^\prime}f^\lambda)(z,w)=  \int_{\R^{2n}} \frac{\partial}{\partial x_j^\prime}f^\lambda(x^\prime,u^\prime) \, p_{1/2}^\lambda(z-x^\prime,w-u^\prime)\, e^{-i\frac{\lambda}{2} ( x^\prime \cdot w- u^\prime \cdot z)} \, dx^\prime\, du^\prime.$$
Integrating by parts and collecting the terms we get
$$ (-a z_j+b w_j)e^{-\frac{1}{2}L_\lambda}f^\lambda(z,w) = e^{-\frac{1}{2}L_\lambda} A_j f^\lambda(z,w).$$
In a similar way we can also prove the identity
$$ (b z_j+a w_j)e^{-\frac{1}{2}L_\lambda}f^\lambda(z,w) = e^{-\frac{1}{2}L_\lambda} B_j f^\lambda(z,w).$$
Now the lemma follows from these two identities.
\end{proof}

From the lemma we infer that  $ z_j B_\lambda f^\lambda(z,w) = B_\lambda(P_{j,\lambda} f^\lambda)(z,w) $ and $ w_j B_\lambda f^\lambda(z,w) = B_\lambda(Q_{j,\lambda} f^\lambda)(z,w) $ where $ P_{j,\lambda} $ and $ Q_{j,\lambda}$ are first order differential operators with monomial coefficients. The constants occurring in these differential  operators are one of the following: $ a(a^2+b^2)^{-1},\, a^2(a^2+b^2)^{-1},\, b(a^2+b^2)^{-1} $ and $ab(a^2+b^2)^{-1}.$ By inspection we can convince ourselves that these constants are bounded by $ (1+|\lambda|)\, e^{|\lambda|}.$ More generally, for $ \alpha \in \mathbb{N}^{2n} $ with $ |\alpha|= N$ we have 
$$ \zeta^\alpha\,  B_\lambda f^\lambda(\zeta) = B_\lambda(P f^\lambda)(\zeta ),\,\,\, P = \sum_{|\mu|+|\nu|=N} c_{\mu,\nu}(\lambda)\, (x,u)^\mu \partial_{(x,u)}^\nu$$  where the constants are bounded by $ (1+|\lambda|)^n\, e^{N|\lambda|}.$ As before this gives the estimate
\begin{equation}\label{pointwise-barg-two}
 |B_\lambda f^\lambda(\zeta)| \sqrt{w_\lambda(\zeta)} \leq C  c_\lambda^2\,  (1+\lambda^2)^{N/2} \, e^{N|\lambda|}\,(1+|\zeta|)^{-N}  \sum_{|\mu|+|\nu|=N} \| (x,u)^\mu \partial_{(x,u)}^\nu f^\lambda \|_2.
 \end{equation}
 Once again as $ f $ is Schwartz, the last term in the above inequality can be estimated independent of $ \lambda.$  Combining \eqref{pointwise-barg-one} and \eqref{pointwise-barg-two}  we get
 
  \begin{prop}\label{point-barg} Let $ f $ be a Schwartz function on $ \He^n.$ Then for any $ m, N \in \mathbb{N} $ we have
  $$ |B_\lambda f^\lambda(\zeta)| \sqrt{w_\lambda(\zeta)} \leq C_{m,N}(f)\,  c_\lambda\, (1+\lambda^2)^{-m} \, e^{N|\lambda|}\,(1+|\zeta|)^{-N} $$
 \end{prop}
   \vskip.1truein
 \begin{rem}The decay in the $ \zeta $ variable seems to be optimal. But we do not know if the exponential growth in  $ \lambda $ can be improved as it is a result of the behaviour of the coefficients $ a $ and $ b$ in the previous lemma.
 \end{rem}

We have two candidates for the modulation spaces  on $ \He^n .$  The most natural definition is the one where the Schr\"odinger representation  $ \pi_1 $ of $ \He^n $  is replaced by the representation $ \Pi $ of the group $G_n.$ Thus we may define $ M_0^{p,q}(\He^n) $ as the space of all tempered distributions $ f $ such that $ B_{\He}f(\zeta,0) =  \langle Tf, \Pi(\zeta,0) h_{1/2} \rangle, \zeta = \xi+i\eta $ belongs to the mixed norms space $ L^{p,q}(\R^{2n} \times \R^{2n}).$ 
In view of \eqref{esti-point} it follows that $ L^2(\He^n) \subset M_0^{2,2}(\He^n).$ However, it is  not clear  if $ L^2(\He^n) = M_0^{2,2}(\He^n).$ 
The identity \eqref{matrix-planch}  suggests that we can also define the modulation spaces in terms of the modified matrix coefficients $ \langle Tf, \Pi(\zeta,s) h_{1/2} \rangle $ and their $ L^2 $ norms with respect to the $s$ variable. \\

Let $ L^{p,q}(\R^{2n} \times \R^{2n}, L^2(\R)) $ stand for the space of functions $ F(\xi,\eta, \cdot) $ on $ \R^{2n} \times \R^{2n} $ taking values in $ L^2(\R)$ equipped with the obvious norm.  We say that a tempered distribution $ f $ belongs to $ M_1^{p,q}(\He^n) $ if  the function  $\langle Tf, \Pi(\xi+i\eta,s) p_{1/2} \rangle $ belongs to $ L^{p,q}(\R^{2n} \times \R^{2n}, L^2(\R)) .$ In view of \eqref{matrix-modi} this  is the same as saying that the function
\begin{equation}\label{definition} \|\langle Tf, \Pi(\xi+i\eta,\cdot) p_{1/2} \rangle\|_2  = \left(\int_{-\infty}^\infty | B_\lambda f^\lambda(\xi+i\eta)|^2 \, w_\lambda(\xi+i\eta) \, d\lambda \right)^{1/2} 
\end{equation}
belongs to $ L^{p,q}(\R^{2n} \times \R^{2n}).$ With this definition it is clear that $ M_1^{2,2}(\He^n) = L^2(\He^n).$ Moreover, we also have the inclusion $ M_1^{p,q}(\He^n) \subset M_0^{p,q}(\He^n)$  which is a consequence of \eqref{esti-point-one}. However, we are not able to prove the completeness of  $ M_1^{p,q}(\He^n) .$ \\

We get another definition by interchanging the roles of the mixed norms in $ (\xi,\eta) $ and the $ L^2 $ norm in $ \lambda.$  This amounts to requiring that  the $ L^{p,q}(\R^{2n} \times \R^{2n}) $ norm of $ B_\lambda f^\lambda(\zeta) \sqrt{w_\lambda(\zeta)} $ is square integrable with respect to $ \lambda.$ Note that this does not alter the $ p=q=2$ case but seems to be  a reasonable definition as we plan to demonstrate now.\\

 In order to make the definition $ M^{p,q}(\He^n) ,$ let examine the modified matrix coefficients more closely. By polarising the identity \eqref{matrix-planch} and using the fact that $\langle Tf, \Pi(\zeta,s) p_{1/2} \rangle = \langle f, \Pi(\zeta,s) Tp_{1/2} \rangle $ we get
\begin{equation}\label{matrix-parse}
\int_{\C^{2n}} \int_{-\infty}^\infty \langle f, \Pi(\zeta,s) Tp_{1/2} \rangle\, \langle \Pi(\zeta,s) Tp_{1/2}, g\rangle\, d\zeta\, ds  =  c_n^{\prime \prime}\, \int_{\He^n} f(h) \, \overline{g(h)} \,dh.
\end{equation}
The above identity gives us the following inversion formula:
\begin{equation}\label{h-inverse}
 f(h) =c_n \, \int_{\C^{2n}} \int_{-\infty}^\infty \langle f, \Pi(\zeta,s) Tp_{1/2} \rangle\, \Pi(\zeta,s) Tp_{1/2}(h)\, d\zeta\, ds.
\end{equation}
Let us set up some notations. Let $ \mathcal{H}_{p,q} = L^{p,q,2}(\R^{2n} \times \R^{2n} \times \R)$  be the mixed norm space  equipped with the norm $\| F \|_{p,q,2}.$ For a function or a tempered distribution $ F(\xi,\eta,s) $ on $ \R^{2n} \times \R^{2n} \times \R $ let $ F^\lambda (\xi,\eta) $ stand for its inverse Fourier  transform in the last variable.  We then  define 
$$ \widehat{\mathcal{H}}_{p,q} = \{ F \in \mathcal{S}^\prime(\R^{2n} \times \R^{2n} \times \R): F^\lambda(\xi,\eta) \in  \mathcal{H}_{p,q} \} $$  equipped with the norm $   \|F\|_{\widehat{\mathcal{H}}_{p,q}} =  \|F^\lambda\|_{\mathcal{H}_{p,q}} = \| F^\lambda \|_{p,q,2}.$  

\begin{rem} A word of caution about the notation. When $ f $ is a tempered distribution on $ \R^{m+1},$  the partial inverse Fourier transform $ \mathcal{F}^\ast_{m+1} f $ is again a tempered distribution. When  this distribution is given by a function, we write $ f^\lambda(x) $ instead of $ \mathcal{F}^\ast_{m+1} f .$ 
\end{rem}

We need the completeness of the spaces  $ \mathcal{H}_{p,q} $ and $  \widehat{ \mathcal{H}}_{p,q}$ proved in the following lemma.

 \begin{lem}\label{Hpq} Both  $ \mathcal{H}_{p,q} $ and $  \widehat{ \mathcal{H}}_{p,q}$ are Banach spaces. (The norm of $ F \in  \widehat{ \mathcal{H}}_{p,q}$ is the norm of $ F^\lambda $ in $   \mathcal{H}_{p,q}.$)
\end{lem}
\begin{proof} The completeness of  $ \mathcal{H}_{p,q}$ is immediate since it  is a mixed norm space, see \cite{BP}.  To prove the completeness of  $  \widehat{ \mathcal{H}}_{p,q},$ let $ F_k $ be Cauchy. Then by definition $ F_k^\lambda $ is Cauchy in $  \mathcal{H}_{p,q}.$ Hence  $ F_k^\lambda $ converges to some $ F \in  \mathcal{H}_{p,q}.$ It is enough to show that there exists $ G \in \widehat{ \mathcal{H}}_{p,q}$ such that $ F(\zeta,\lambda) = G^\lambda(\zeta).$  In order to show this we make use of the following well known fact. 

Let $ \Phi_\alpha, \alpha \in \mathbb{N}^m $ be the family of Hermite functions on $ \R^m.$ Then every tempered distribution $ F $ on $ \R^m $ has a formal expansion 
$ F = \sum_{\alpha \in \mathbb{N}^m}   \langle F, \Phi_\alpha\rangle \, \Phi_\alpha $
in the sense that for any $ \varphi \in \mathcal{S}(\R^m) $ the action of $ F $ on $ \varphi $ is given by the absolutely convergent series
$$ (F, \varphi) = \sum_{\alpha \in \mathbb{N}^m}   \langle F, \Phi_\alpha\rangle \,  \, \langle \varphi, \Phi_\alpha \rangle. $$
This is due to the fact that $ F $ is tempered if and only if $| \langle F, \Phi_\alpha\rangle | \leq C\, (1+|\alpha|)^N $ for some $ N $ whereas $ \varphi $ is Schwartz if and only if 
$| \langle \varphi, \Phi_\alpha \rangle | \leq C_j\, (1+|\alpha|)^{-j} $ for every $ j.$ 

Since the tempered distribution $ F $ is given by a function $ F(\xi,\eta,\lambda) $ we have the expansion
$$ F(\xi,\eta,\lambda) = \sum_{j=0}^\infty \sum_{\alpha \in \mathbb{N}^{4n}}   \langle F, \Phi_\alpha \otimes \Phi_j \rangle  \,  \Phi_\alpha (\xi,\eta)\,\Phi_j(\lambda). $$
As $ F $ is tempered, $   \langle F, \Phi_\alpha \otimes \Phi_j \rangle$ have polynomial growth and hence the formal expansion 
$$ G(\xi,\eta, t) = \sum_{j=0}^\infty \sum_{\alpha \in \mathbb{N}^{4n}}  (-i)^j\,  \langle F, \Phi_\alpha \otimes \Phi_j \rangle  \,  \Phi_\alpha (\xi,\eta)\,\Phi_j(t) $$
defines a tempered distribution. Since $ \widehat{\Phi_j}(\lambda) = (-i)^j\, \Phi_j(\lambda) $ it is immediate that $ G^\lambda(\xi,\eta) = F(\xi,\eta,\lambda) $ proving that $ G \in \widehat{ \mathcal{H}}_{p,q}.$ 
\end{proof}

For the sake of simplicity of notation, let us set
$ Vf(\zeta,s) =\langle f, \Pi(\zeta,s) Tp_{1/2} \rangle.$  We also write $ V_\lambda f(\zeta) $ in place of $ (Vf)^\lambda(\zeta).$
 We note that  $ V $ takes $ L^2(\He^n) $ into $ \widehat{\mathcal{H}}_{2,2} = L^2(\R^{2n} \times \R^{2n} \times \R).$  This is a consequence of \eqref{matrix-modi} and the fact that  $B_\lambda $ takes $ L^2(\R^{2n}) $ isometrically into $ L^2(\R^{2n}\times \R^{2n}).$  This suggests that we make the following definition of  modulation spaces on $ \He^n.$\\
 
We say that a tempered distribution $ f $ belongs to $ M^{p,q}(\He^n) $ if $ Vf(\zeta,s) $  belongs to $ \widehat{ \mathcal{H}}_{p,q}.$ In other words the function $ V_\lambda f(\zeta) \in L^{p,q}(\R^{2n} \times \R^{2n}) $ and 
the integral 
$$ \|f\|_{(p,q)}^2 =:  \int_{-\infty}^\infty \| V_\lambda f \|_{p,q}^2\, d\lambda  =  \int_{-\infty}^\infty \| B_\lambda f^\lambda  \sqrt{w_\lambda}\, \|_{p,q}^2\, d\lambda< \infty $$
is finite. We equip $ M^{p,q}(\He^n) $ with the norm $ \|f\|_{(p,q)} $ defined above.\\
 
 \begin{rem}\label{rem} Observe that when $ f \in M^{p,q}(\He^n),$  the function $ B_\lambda f^\lambda(\zeta) \sqrt{w_\lambda(\zeta)} $ belongs to the mixed norm space $ L^{p,q}(\R^{2n}\times \R^{2n})$  for  a.e $ \lambda.$ Hence $ f^\lambda \in M_\lambda^{p,q}(\R^{2n}) $ and we may be tempted to define  $ M^{p,q}(\He^n) $ by replacing $ p_{1/2} $ in the definition by any Schwartz class function $ g $ on $ \He^n.$ Thus  we may say that $ f \in M^{p,q}(\He^n) $ if and only if  the inverse Fourier transform of  $ \langle Tf, \Pi(\zeta,s)g \rangle $ in the $ s $ variable belongs to $ M_\lambda^{p,q} (\R^{2n}) $ in such a way that
 $$ \int_{-\infty}^\infty  c_\lambda\, \| \langle f^\lambda, \Pi_\lambda(\zeta)g^\lambda\rangle\|_{p,q}^2\, d\lambda <\infty.$$
 This is certainly a better definition if  we can show that it is independent of the choice of $ g.$ Unfortunately, it is not the case- if the above integral is finite for a $ g $ there is no guarantee that it will be finite for  other choices of $ g.$ Moreover, only with $ g = p_{1/2} $ we can  get the desired identification $ M^{2,2}(\He^n)= L^2(\He^n).$\\
 \end{rem}

 The identity \eqref{h-inverse} suggests that we define an operator $ \widetilde{V} $ acting on  functions $ F $ defined on $ \R^{2n} \times \R^{2n} \times \R$  as follows:
\begin{equation}\label{vtilde-def} \widetilde{V}F(h) =\int_{\C^{2n}} \int_{-\infty}^\infty F(\zeta,s)\, \Pi(\zeta,s) Tp_{1/2}(h)\, d\zeta\, ds .
\end{equation}
It is then clear from \eqref{h-inverse} that $ \widetilde{V} ( Vf)  = f .$  We will  show that $ \widetilde{V}F $ makes sense for all $ F \in  \widehat{\mathcal{H}}_{p,q} .$  Consider the functional 
\begin{equation}\label{vtilde-def-one} \langle \widetilde{V}F, g \rangle = \int_{\C^{2n}} \int_{-\infty}^\infty F^\lambda(\zeta) \, \overline{V_\lambda g(\zeta) }\, \, d\lambda  \, d\zeta  
\end{equation}
where $ g $ is a Schwartz class function  on $ \He^n.$ From \eqref{matrix-modi} we have 
$$ \langle g, \Pi(\zeta,s) Tp_{1/2} \rangle = c_n \int_{-\infty}^\infty e^{-i\lambda s} B_\lambda g^\lambda(\zeta)\, \sqrt{w_\lambda(\zeta)}\, d\lambda$$
and hence we can use the estimate proved in Proposition \ref{point-barg} to get the estimate
$$ |\langle \widetilde{V}F, g \rangle | \leq  C_N(g)  \int_{-\infty}^\infty \| F^\lambda\|_{p,q}\, e^{N|\lambda|}\, d\lambda $$
where $ C_N(g) $ is a Schwartz seminorm. Therefore, if we assume that $ F^\lambda $ is compactly supported in $ \lambda,$ then $ \widetilde{V}F $ is seen to be a tempered distribution. We can take 
\eqref{vtilde-def-one} as an equivalent definition of $ \widetilde{V}F.$
\\

\begin{lem} \label{v-vtilde}The operator   $ V \circ \widetilde{V} $ initially defined on the subspace of band limited functions in the $ s$ variable has an extension to  the whole of  $  \widehat{\mathcal{H}}_{p,q} .$ Moreover, $ \widetilde{V} ( Vf)  = f $ holds for all $ f \in M^{p,q}(\He^n).$
\end{lem}
\begin{proof} For $ F \in   \widehat{\mathcal{H}}_{p,q} $ which is band limited in the $ s $ variable, we consider $ (V\circ \widetilde{V})F.$ This given by the integral
\begin{equation}\label{vvtilde} \langle \widetilde{V}F, \Pi(\zeta,s)Tp_{1/2}\rangle = \int_{\C^{2n}} \int_{-\infty}^\infty F(\zeta^\prime,s^\prime)\, \langle \Pi(\zeta^\prime,s^\prime) Tp_{1/2}, \Pi(\zeta,s)Tp_{1/2}\rangle \, d\zeta^\prime \, ds^\prime.
\end{equation}
The above integral has to be interpreted as the action of the tempered distribution $ \widetilde{V}F $ acting on the Schwartz function $\Pi(\zeta,s)Tp_{1/2}$ on $ \He^n,$ see Lemma \ref{schwartz}.
As $ \Pi(\zeta,s) $ is defined in terms of $ \Pi_\lambda(\zeta,s)$ we see that
$$\langle \Pi(\zeta^\prime,s^\prime) Tp_{1/2}, \Pi(\zeta,s)Tp_{1/2}\rangle  =  \int_{-\infty}^\infty  c_\lambda\,\,  \langle \Pi_\lambda(\zeta^\prime,s^\prime) p_{1/2}^\lambda, \Pi_\lambda(\zeta,s)p_{1/2}^\lambda\rangle  \, d\lambda. $$
As $ \Pi_\lambda $ is a representation of the group $ G_n $  and $ \Pi_\lambda (\zeta^\prime,s^\prime)^\ast = \Pi_\lambda (-\zeta^\prime,-s^\prime)$ the above becomes
$$\langle \Pi(\zeta^\prime,s^\prime) Tp_{1/2}, \Pi(\zeta,s)Tp_{1/2}\rangle  =  \int_{-\infty}^\infty c_\lambda\,  \langle  p_{1/2}^\lambda, \Pi_\lambda((-\zeta^\prime,-s^\prime)(\zeta,s))p_{1/2}^\lambda\rangle  \, d\lambda. $$
Recalling the group law and the definition of $ \Pi_\lambda $ we see that the above integral becomes
$$  \int_{-\infty}^\infty e^{i\lambda(s-s^\prime)} e^{i\frac{1}{2} (\lambda \Re[\zeta,\bar{\zeta^\prime}]- (\coth \lambda) \Im(\zeta \cdot \bar{\zeta^\prime}))}c_\lambda\, \langle  p_{1/2}^\lambda, \Pi_\lambda(\zeta-\zeta^\prime)p_{1/2}^\lambda\rangle  \, d\lambda. $$
Thus the inverse Fourier transform of $\langle \Pi(\zeta^\prime,s^\prime) Tp_{1/2}, \Pi(\zeta,s)Tp_{1/2}\rangle$ in the $ s^\prime $ variable evaluated at $ \lambda $ is equal to
$$e^{i\lambda s} \,e^{i\frac{1}{2} (\lambda \Re[\zeta,\bar{\zeta^\prime}]- (\coth \lambda) \Im(\zeta \cdot \bar{\zeta^\prime}))} \, c_\lambda\,  \langle  p_{1/2}^\lambda, \Pi_\lambda(\zeta-\zeta^\prime)p_{1/2}^\lambda\rangle.$$
Rewriting the integral on the right hand side of \eqref{vvtilde} using Plancherel in the $ s^\prime $ variable we can express the function $ G(\zeta,s) =  \langle \widetilde{V}F, \Pi(\zeta,s)Tp_{1/2}\rangle$ as
\begin{equation}\label{vvtildelambda}
\int_{\C^{2n}} \int_{-\infty}^\infty F^\lambda(\zeta^\prime)\,e^{i\lambda s} \,e^{i\frac{1}{2} (\lambda \Re[\zeta,\bar{\zeta^\prime}]- (\coth \lambda) \Im(\zeta \cdot \bar{\zeta^\prime}))} \, c_\lambda\,  \langle  p_{1/2}^\lambda, \Pi_\lambda(\zeta-\zeta^\prime)p_{1/2}^\lambda\rangle d\zeta^\prime \, d\lambda.
\end{equation}
By defining $  K_\lambda(\zeta) =  c_\lambda\,   \langle p_{1/2}^\lambda, \Pi_\lambda(\zeta)p_{1/2}^\lambda\rangle,$ the above leads to the inequality
$$ |G^\lambda(\zeta)| \leq  \int_{\C^{2n}} |F^\lambda(\zeta^\prime)|\, \, | K_\lambda(\zeta-\zeta^\prime)|\,  d\zeta^\prime \,.$$
We claim  that $ K_\lambda \in L^1(\R^{2n}\times \R^{2n}),$ uniformly in $ \lambda .$ Assuming the claim for a moment,  by Young's inequality for mixed norm spaces we get
$$ \| G^\lambda\|_{p,q}  \leq \| K_\lambda \|_{1,1}\, \|F^\lambda\|_{p,q} \leq C\, \|F^\lambda\|_{p,q} . $$
Once we have the above, taking $ L^2 $ norm in the $ \lambda$ variable  we get the estimate
$$ \| (V \circ \widetilde{V})F \|_{\widehat{\mathcal{H}}_{p,q}} \leq C\,  \|F\|_{\widehat{\mathcal{H}}_{p,q}} .$$ 
The important point to note is that the constant $ C $ in the above inequality is independent of the support of $ F^\lambda.$  Let $ \varphi $ be a Schwartz function on $ \R $ whose Fourier transform is supported on $ [-2,2] $ and  identically  1 on [-1,1].  Defining $ \varphi_\epsilon(t) = \epsilon^{-1} \varphi(\epsilon^{-1}t) $ let
$$ F_\epsilon(\zeta,s) =: F \ast_3 \varphi_\epsilon(\zeta,s) = \int_{-\infty}^\infty F(\zeta,s-t)\, \varphi_\epsilon(t)\, dt.$$
Then $ F_\epsilon $ is band limited and converges to $ F $ in $\widehat{\mathcal{H}}_{p,q}.$  We can then define $  (V \circ \widetilde{V})F $ as the limit of $ (V \circ \widetilde{V})F_\epsilon.$ \\

Returning to our claim, the required property of the kernel $ K_\lambda $ is proved by explicit calculation. Indeed, in view of \eqref{matrix-barg}, as $ B_\lambda p_{1/2}^\lambda = c_\lambda,$ we obtain
\begin{equation}\label{K lambda explicit expression}  K_\lambda(\xi,\eta) =  d_n \, c_\lambda\, \langle  p_{1/2}^\lambda, \Pi_\lambda(\xi,\eta) p_{1/2}^\lambda \rangle =  c_\lambda^2\,  \,e^{-\frac{1}{4}\lambda (\coth \lambda)(|\xi|^2+|\eta|^2)}\, e^{\frac{1}{2}\lambda [\xi,\eta]}.
\end{equation}
With $ \xi =(x,u)$ and $ \eta = (y,v)$ we have $ [\xi,\eta] = (u \cdot y-v \cdot x)$ and hence we can rewrite 
$$ K_\lambda(\xi,\eta) = c_\lambda^2\,  \, e^{-\frac{1}{8} \lambda \left( \coth \lambda +1 \right)( (x+v)^2+(u-y)^2)}\, e^{-\frac{1}{8}\lambda \left( \coth \lambda -1 \right)( (x-v)^2+(u+y)^2)}.$$ 
By introducing the new variables $ x^\prime = x+v, v^\prime = x-v, u^\prime = u-y, y^\prime = u+y , \, \xi^\prime =(x^\prime,u^\prime) $ and $ \eta^\prime =(y^\prime,v^\prime)$ we have
$$ \int_{\R^{2n}}\int_{\R^{2n}} K_\lambda(\xi,\eta) d\xi\, d\eta= c_\lambda^2\,  \, \int_{\R^{2n}}\int_{\R^{2n}}e^{-\frac{1}{8} \lambda \left( \coth \lambda +1 \right)|\xi^\prime|^2}\, e^{-\frac{1}{8}\lambda \left( \coth \lambda -1 \right) |\eta^\prime|^2}\, d\xi^\prime d \eta^\prime.$$
Evaluating the integral on the right hand side we see that 
$$ \int_{\R^{2n}}\int_{\R^{2n}} K_\lambda(\xi,\eta) d\xi\, d\eta= c_n\,  \lambda^{2n} (\sinh \lambda)^{-2n} \big( \lambda^2 (\coth\lambda+1)(\coth \lambda-1)\big)^{-n} =c_n. $$
This proves the required property of $ K_\lambda.$\\

Finally, we need to show that $ \widetilde{V}(Vf) = f $ for all $ f \in M^{p,q}(\He^n).$ With notation as above, note that $ V(f\ast_3 \varphi_\epsilon) = Vf \ast_3 \varphi_\epsilon$ and hence  $ f \ast_3 \varphi_\epsilon $ converges to $ f $ in $M^{p,q}(\He^n).$ We can therefore assume that $ f $ is band limited in the $ t $ variable so that $ Vf(\zeta,s) $ is band limited in the $ s $ variable.  To prove  $ \widetilde{V}(Vf) = f $  let us check the action of the tempered distribution $ \widetilde{V}(Vf) $ on a Schwartz function $ g.$ Thus
$$ \langle \widetilde{V}(Vf), g \rangle  = \int_{\C^{2n}} \int_{-\infty}^\infty  V_\lambda f(\zeta)\,  \, \overline{V_\lambda g(\zeta) }\, \, d\lambda  \, d\zeta  .$$
Interchanging the order of integration, the integral over $ \C^{2n} $ can be evaluated using the properties of the operators $ V_h^\lambda $ proved  in Lemma \ref{polarised-id}. Indeed, by recalling the respective definitions and taking $ h = p_{1/2}^\lambda $ we easily  calculate that
\begin{equation}\label{vlambda} V_\lambda f(\zeta) =c_n\, \sqrt{c_\lambda}\, e^{-i \frac{\lambda}{2} (\coth \lambda) \xi \cdot \eta }\,  V_h^\lambda f^\lambda (\zeta) 
\end{equation}  
and a similar expression for $ V_\lambda g.$ By appealing to the orthogonality relation \eqref{orthogonality relation} we
$$\int_{\C^{2n}}   V_\lambda f(\zeta)\,  \, \overline{V_\lambda g(\zeta) }\,   \, d\zeta  = c_n \, c_\lambda^{-1}\,  \langle p_{1/2}^\lambda, p_{1/2}^\lambda \rangle \,\int_{\R^{2n}} f^\lambda(x,u)\, \overline{g^\lambda(x,u)}\, dx\,du.$$
A simple calculation using the explicit expression for $ p_{1/2}^\lambda $ shows that $\langle p_{1/2}^\lambda, p_{1/2}^\lambda \rangle = d_n c_\lambda $ and hence 
$$\int_{\C^{2n}}   V_\lambda f(\zeta)\,  \, \overline{V_\lambda g(\zeta) }\,   \, d\zeta  = c_n  \,\int_{\R^{2n}} f^\lambda(x,u)\, \overline{g^\lambda(x,u)}\, dx\,du.$$ 
Integrating the above with respect $ \lambda $  we obtain  $\langle \widetilde{V}(Vf), g \rangle = c_n \,\langle f, g\rangle $ proving the result.
\end{proof}
From the above lemma we infer that the operator $ T = V \circ \widetilde{V} $ is bounded on $ \widehat{\mathcal{H}}_{p,q} .$ We now show that $ T $ is self-adjoint, a property which will be used later in computing the dual of $ M^{p,q}(\He^n).$

\begin{lem}\label{v-vtilde-adjoint}  For any $ F \in  \widehat{\mathcal{H}}_{p,q} $ and $ G \in  \widehat{\mathcal{H}}_{p^\prime,q^\prime} $ let  $ F_0 = TF, G_0 = TG.$ Then
$$ \int_{-\infty}^\infty \langle F_0^\lambda, G^\lambda \rangle \, d\lambda =  \int_{-\infty}^\infty \langle F^\lambda, G_0^\lambda \rangle \, d\lambda.$$ 
\end{lem}
\begin{proof} We have already computed $ F_0(\zeta,s) $ in \eqref{vvtilde} in the proof of Lemma \ref{v-vtilde}. From \eqref{vvtildelambda} we obtain the following expression:
$$ F_0^\lambda(\zeta) = \int_{\C^{2n}}  F^\lambda(\zeta^\prime) \,e^{i\frac{1}{2} (\lambda \Re[\zeta,\bar{\zeta^\prime}]- (\coth \lambda) \Im(\zeta \cdot \bar{\zeta^\prime}))} \, c_\lambda\,  \langle  p_{1/2}^\lambda, \Pi_\lambda(\zeta-\zeta^\prime)p_{1/2}^\lambda\rangle \, d\zeta^\prime \, .$$
Integrating $ F_0^\lambda(\zeta) $ against $ \overline{G^\lambda(\zeta)}, $ changing the order of integration and simplifying we can easily check that  $\langle F_0^\lambda, G^\lambda \rangle =\langle F^\lambda, G_0^\lambda \rangle$ which proves the lemma.
\end{proof}

With the above lemmas in hand, we can now establish some basic properties of the modulation spaces $ M^{p,q}(\He^n).$

\begin{thm}\label{Msp invariant under U} $ M^{p,q}(\He^n) $ are Banach spaces for  any $ 1 \leq p,q \leq \infty $ which are invariant under $ \Pi(\zeta^\prime,s^\prime)$ for all  $ (\zeta^\prime,s^\prime) \in \C^{2n} \times \R.$ Moreover, $ M^{p,p}(\He^n) $ are invariant under $ U.$
\end{thm}
\begin{proof}  Let  $ (f_k) $ be Cauchy in $ M^{p,q}(\He^n) $ so that $ F_k = Vf_k $ is  Cauchy in  $\widehat{\mathcal{H}}_{p,q} .$  Then by the completeness of $ \widehat{\mathcal{H}}_{p,q} $  there exists $ F \in \widehat{\mathcal{H}}_{p,q} $  such that 
$ F_k  \rightarrow F $ in $\widehat{\mathcal{H}}_{p,q} .$  In view of Lemma \ref{v-vtilde},  $ f = \widetilde{V}F \in M^{p,q}(\He^n) $ and hence
$$ \| V(f_k- f) \|_{\widehat{\mathcal{H}}_{p,q}} =  \| V\circ \widetilde{V} (F_k-F)\|_{\widehat{\mathcal{H}}_{p,q}} \leq C \| F_k-F\|_{\widehat{\mathcal{H}}_{p,q}}.
$$
which shows that $ f_k \rightarrow f $ in $ M^{p,q}(\He^n) $  proving the completeness. To prove the invariance of $ M^{p,q}(\He^n) $ under $ \Pi(\zeta^\prime,s^\prime),$ let $ g = \Pi(\zeta^\prime,s^\prime)f $ and consider
$$ Vg(\zeta,s) = \langle  \Pi(\zeta^\prime,s^\prime)f, \Pi(\zeta,s) Tp_{1/2} \rangle .$$
 As $ \Pi $ is defined in terms of $ \Pi_\lambda $  it is easy to check that
$$ Vg(\zeta,s) = c_n  \int_{-\infty}^\infty e^{-i\lambda s}\, e^{i\frac{1}{2} (\lambda \Re[\zeta,\bar{\zeta^\prime}]- (\coth \lambda) \Im(\zeta \cdot \bar{\zeta^\prime}))} \, \sqrt{c_\lambda}\, \langle  f^\lambda, \Pi_\lambda(\zeta-\zeta^\prime) p_{1/2}^\lambda \rangle \, d\lambda. $$
In other words, we have the equality
$$ \int_{-\infty}^\infty Vg(\zeta,s)\,e^{i\lambda s} ds = e^{i\frac{1}{2} (\lambda \Re[\zeta,\bar{\zeta^\prime}]- (\coth \lambda) \Im(\zeta \cdot \bar{\zeta^\prime}))} \, \int_{-\infty}^\infty Vf(\zeta-\zeta^\prime,s)\,e^{i\lambda s} ds$$
 from which follows the stated invariance of $ M^{p,q}(\He^n).$ To prove the invariance of $ M^{p,p}(\He^n) $ under $ U $ we make use of the relation
$$ \langle f^\lambda, \Pi_\lambda(\xi,\eta)g^\lambda \rangle = \langle f^\lambda,  U_\lambda^\ast \circ \Pi_\lambda(-\eta,\xi) \circ U_\lambda g^\lambda\rangle = \langle U_\lambda f^\lambda, \Pi_\lambda(-\eta,\xi) U_\lambda g^\lambda\rangle .$$
We also  use the fact that $ U_\lambda p_{1/2}^\lambda = p_{1/2}^\lambda $ which follows from the relation $ U \circ B_\lambda = B_\lambda \circ U_\lambda.$ As $ U $ is defined in terms of $ U_\lambda $ we have
$$ \langle T (Uf), \Pi(\zeta,s) p_{1/2} \rangle = c_n  \int_{-\infty}^\infty e^{-i\lambda s}\, \sqrt{c_\lambda}\, \langle U_\lambda f^\lambda, \Pi_\lambda(\xi,\eta) U_\lambda p_{1/2}^\lambda \rangle \, d\lambda  =\langle Tf, \Pi(-i\zeta,s) p_{1/2} \rangle.$$
This proves that $ Uf \in M^{q,p}(\He^n) $ whenever $ f \in M^{p,q}(\He^n).$
\end{proof}

\vskip 0.1truein
\begin{rem}
We can rewrite the equality $\langle T (Uf), \Pi(\zeta,s) p_{1/2} \rangle=\langle Tf, \Pi(-i\zeta,s) p_{1/2} \rangle$ obtained in the proof of Theroem \ref{Msp invariant under U} as
$$ V(Uf)(\zeta, s)=V(f)(-i\zeta, s). $$
The above equality is the analogue of the fundamental identity of the time-frequency analysis. 
\end{rem}

We have seen inclusion between twisted modulation spaces in Theorem \ref{lambda Msp prop}. In the following theorem we lift this property to the case of modulation spaces on the Heisenberg group.

\begin{thm}
Let $ 1\leq p_1\leq p_2\leq\infty$ and $ 1\leq q_1\leq q_2\leq\infty$. Then we have the inclusion
    $$ M^{p_1,q_1}(\He^{n})\subseteq M^{p_2,q_2}(\He^{n}).$$
\end{thm}
\begin{proof}
Let $ f\in M^{p_1,q_1}(\He^{n})$  so that $ V_\lambda f\in L^{p_1,q_1}(\R^{2n}\times \R^{2n})$.  Since we have the relation  $ |V_\lambda f(\zeta)|= \sqrt{c_\lambda}\,|V_{p_{1/2}^\lambda}^\lambda f^\lambda(\zeta)|$, by using \eqref{STFT pointise estimate} with $ g=p_{1/2}^\lambda$ we get
$$ \|V_g^\lambda f^\lambda\|_{p_2,q_2}\leq c_n\,c_\lambda\, \|V_g^\lambda f^\lambda\|_{p_1,q_1}\|V_g^\lambda g\|_{r,s}, $$
where $ r,s$ are as in the proof of Theorem \ref{lambda Msp prop}.  Once we show that $ c_\lambda\,  \|V_g^\lambda g\|_{r,s} $ is a bounded function of $ \lambda$ for any $ 1\leq r,s\leq \infty$, the desired   inclusion follows from the above inequality by taking the $ L^2 $ norm in the $ \lambda $ variable. 
Now recall that we have already computed the function  $c_\lambda\,|V_g^\lambda g(\xi, \eta)| = c_\lambda\,|\langle p_{1/2}^\lambda, \Pi_\lambda(\xi, \eta)p_{1/2}^\lambda\rangle|  $ in \eqref{K lambda explicit expression}. It is explicitly given by
$$ K_\lambda(\xi,\eta) =  c_\lambda^2 \,  \,e^{-\frac{1}{4}\lambda (\coth \lambda)(|\xi|^2+|\eta|^2)}\, e^{\frac{1}{2}\lambda [\xi,\eta]}.
$$
Writing $ \xi= (x,u)$ and $ \eta= (y,v)$, the expression for $ K_\lambda$  takes the form
$$  K_\lambda(\xi,\eta)  = c_\lambda^2 \, e^{-\frac{1}{4}\lambda(\coth\lambda)(x^2+u^2+y^2+v^2)} e^{\frac{1}{2}\lambda(u\cdot y-v\cdot x)}.
$$ 
An easy calculation evaluating the Fourier transform of a Gaussian shows that 
$$ \int_{\R^n} e^{-\frac{1}{4} r\lambda(\coth\lambda)x^2}e^{-\frac{1}{2} r\lambda \,v\cdot x}\, dx= c_{n,r}\, (\lambda^{-1}\tanh\lambda)^{n/2}\, e^{\frac{1}{4}r\lambda (\tanh\lambda)v^2}.
$$
With a similar calculation for the $L^r $ norm in the $ u $ variable we obtain
$$ \left(\int_{\R^{2n}}\, K_\lambda(\xi,\eta)^r \, dx\,du\right)^{1/r} = c_{n,r}\, \, ( \lambda^{-1} \tanh\lambda)^{n/r}e^{-\frac{\lambda}{4}(\coth\lambda-\tanh\lambda)(y^2+v^2)}.
$$
Now by calculating the $ L^s$ norm of the above in the $ \eta=(y,v)$ variable, we get
$$ \|K_\lambda\|_{r,s}= c_{n,r,s}\, \lambda^{-n/s}\,  \,c_\lambda^2\,( \lambda^{-1} \,\tanh\lambda)^{n/r} (\coth\lambda-\tanh\lambda)^{-n/s}.
$$
The right hand side in the above expression simplifies to a constant multiple of
$$ c_{n,r,s}\,  (\lambda/\sinh\lambda)^{2n(1-1/s)}\,(\lambda /\tanh\lambda)^{-n/r}(\lambda\coth\lambda)^{n/s}. $$
Since $ \tanh \lambda $ and $ \coth \lambda $ tend to 1 as  $ \lambda \rightarrow \infty, $ noting that $ 1-1/s >0 ,$ the exponential growth of $ \sinh \lambda $ allows us to conclude that the above is bounded.
\end{proof}
\vskip0.1truein

In the next theorem we  identify the dual space of $ M^{p,q}(\He^n).$ The result is again easy consequence of the corresponding result  for the twisted modulation spaces.

\vskip0.1truein

\begin{thm} For any $ 1 \leq p, q < \infty,$ the dual of $ M^{p,q}(\He^n)$ is $ M^{p^\prime,q^\prime}(\He^n).$
\end{thm}
\begin{proof} In proving this theorem we make use of the fact that $ M_\lambda^{p,q}(\R^{2n})^\ast = M_\lambda^{p^\prime,q^\prime}(\R^{2n}).$ Recall that the duality between $ M_\lambda^{p,q}(\R^{2n})$ and $M_\lambda^{p^\prime,q^\prime}(\R^{2n})$ is given by
$$ \langle \varphi, \psi \rangle_\lambda =  c_\lambda\,  \int_{\R^{2n}}\int_{\R^{2n}} \langle \varphi, \Pi_\lambda(\zeta)p_{1/2}^\lambda \rangle \, \overline{\langle \psi, \Pi_\lambda(\zeta)p_{1/2}^\lambda \rangle}\, d\xi\, d\eta $$
where the integral on the right hand side is the duality between the corresponding mixed norm spaces. Given $ f \in M^{p,q}(\He^n)$ and $ g \in M^{p^\prime,q^\prime}(\He^n) ,$ we define  the duality bracket
$$ \langle f, g\rangle_0 = \int_{-\infty}^\infty  \langle f^\lambda, g^\lambda \rangle_\lambda\, d\lambda.$$
By H\"older and Cauchy-Schwarz we  get the estimate $ |\langle f, g\rangle_0 | \leq \| f\|_{(p,q)}\, \|g \|_{(p^\prime,q^\prime)}.$ Thus for any $ g \in M^{p^\prime,q^\prime}(\He^n) $ the functional $ \Lambda_g(f) =\langle f, g\rangle_0$  belongs to $ M^{p,q}(\He^n)^\ast.$   This proves the inclusion $ M^{p^\prime,q^\prime}(\He^n) \subset M^{p,q}(\He^n)^\ast. $ In order to prove the theorem we need to show that every member of $ M^{p,q}(\He^n)^\ast $ is of the form for some $ g \in M^{p^\prime,q^\prime}(\He^n).$

We start with the following observation.   For any function $ F $ on $ \C^{2n}\times \R ,$ let $ \mathcal{F}_c F $ stand for its Fourier transform in the last variable. Then as we have noted in Lemma \ref{Hpq}, $ \mathcal{F}_c^\ast : \widehat{\mathcal{H}}_{p,q}  \rightarrow \mathcal{H}_{p,q} $ is an isometric isomorphism. Hence $ \Lambda \rightarrow  \Lambda \circ \mathcal{F}_c^\ast $ sets up a one to one correspondence between the duals of  $\mathcal{H}_{p,q} $ and 
$\widehat{\mathcal{H}}_{p,q}.$ Since  $ \mathcal{H}_{p,q} $ is a mixed norm space, we have $ \mathcal{H}_{p,q}^\ast = \mathcal{H}_{p^\prime,q^\prime} $  and hence  $\widehat{\mathcal{H}}_{p,q}^\ast =\widehat{\mathcal{H}}_{p^\prime,q^\prime}.$  If we let $ W_{p,q} $ stand for the image of $ M^{p,q}(\He^n)$ under $ V,$ then it is a closed subspace of $\widehat{\mathcal{H}}_{p,q}.$ Given a bounded linear functional $ \Lambda $ on $ M^{p,q}(\He^n),$  the functional $ \Lambda \circ \widetilde{V} $ belongs to  $ W_{p,q}^\ast $  which by Hahn-Banach theorem gives rise to an element $ \widetilde{\Lambda} \in \widehat{\mathcal{H}}_{p,q}^\ast.$  Hence there exists $ G \in  \widehat{\mathcal{H}}_{p^\prime,q^\prime} $ such that 
 \begin{equation}\label{lambdatilde} \widetilde{\Lambda}(F) = \int_{-\infty}^\infty  \left(\int_{\C^{2n}}   F^\lambda(\zeta)\, \overline{G^\lambda(z)}\, d\zeta \right)\, d\lambda 
 \end{equation}
 for any $ F \in \widehat{\mathcal{H}}_{p,q}.$ By defining  $  g = \widetilde{V}G $ we note that $ g \in M^{p^\prime,q^\prime}(\He^n) $  and $ (V \circ \widetilde{V}) G = Vg.$ For any  $ f \in M^{p,q}(\He^n) $ the function    $ F = Vf $ satisfies $ (V \circ \widetilde{V})F = F .$  Therefore, by taking $ F = Vf $ in \eqref{lambdatilde} and appealing to Lemma \ref{v-vtilde-adjoint} we obtain
$$  \widetilde{\Lambda}(Vf) = \int_{-\infty}^\infty  \langle V_\lambda f, V_\lambda g \rangle\, d\lambda.$$
Since $\Lambda(f) =  \widetilde{\Lambda}(Vf) ,$ recalling the definition of the duality bracket, we obtain 
$$ \Lambda(f) =   \int_{-\infty}^\infty  \langle f^\lambda, g^\lambda \rangle_\lambda\, d\lambda =  \langle f, g\rangle_0.$$
This proves the reverse inclusion $  M^{p,q}(\He^n)^\ast \subset  M^{p^\prime,q^\prime}(\He^n)$ and hence the theorem follows.
\end{proof}

\section*{Acknowledgments}
This work began in the summer of 2024 when both authors were visiting Harish-Chandra Research Institute. They wish to thank the institute for the hospitality and the financial support provided. Biswas is thankful to the National Board for Higher Mathematics (NBHM) under the  Department of Atomic Energy (DAE), Govt. of India for the postdoctoral fellowship with reference number 0204/27/(14)/2023/R\&D-II/11891. The second author is also supported by INSA.

\providecommand{\bysame}{\leavevmode\hbox to3em{\hrulefill}\thinspace}
\providecommand{\MR}{\relax\ifhmode\unskip\space\fi MR }
\providecommand{\MRhref}[2]{%
  \href{http://www.ams.org/mathscinet-getitem?mr=#1}{#2}
}
\providecommand{\href}[2]{#2}

\end{document}